\documentclass[11pt,reqno]{amsart}

 \usepackage{amsfonts,amsmath,amscd,amssymb,amsbsy,amsthm,amstext,amsopn,amsxtra}
 \usepackage{color,fullpage,mathrsfs}
 \usepackage{subfigure}
 \usepackage{verbatim} 
 \usepackage{stmaryrd}
 \usepackage{extarrows}
 \usepackage[all]{xy}
 \usepackage{bm}   
 \usepackage{hyperref}
 \usepackage{tikz}
 \usepackage{tikz-cd}
 \usepackage{enumitem}
 \usetikzlibrary{calc}
\usetikzlibrary{arrows,decorations.pathmorphing,decorations.pathreplacing,positioning,shapes.geometric,shapes.misc,decorations.markings,decorations.fractals,calc,patterns}
\usepackage{dynkin-diagrams}

 \tikzset{>=stealth',
        cvertex/.style={circle,draw=black,inner sep=1pt,outer sep=3pt},
        vertex/.style={circle,fill=black,inner sep=1pt,outer sep=3pt},
        star/.style={circle,fill=yellow,inner sep=0.75pt,outer sep=0.75pt},
        tvertex/.style={inner sep=1pt,font=\scriptsize},
        gap/.style={inner sep=0.5pt,fill=white}}
        
\usepackage{mathtools} 
\usepackage{microtype} 
 \hypersetup{colorlinks,linkcolor=[rgb]{0,0,1},citecolor=[rgb]{1,0,0}}
 
\usepackage{cancel}

 \numberwithin{equation}{section}
\newtheorem{theorem}{Theorem}[section]
\newtheorem{lemma}[theorem]{Lemma}
\newtheorem{proposition}[theorem]{Proposition}

\newtheorem{question}[theorem]{Question}

\theoremstyle{definition}
\newtheorem{definition}[theorem]{Definition}

\theoremstyle{remark}
\newtheorem{remark}[theorem]{Remark}
\newtheorem{remarks}[theorem]{Remarks}
\newtheorem{example}[theorem]{Example}

\newcommand{\one}{\ensuremath{(\mathrm{i})}}
\newcommand{\two}{\ensuremath{(\mathrm{ii})}}
\newcommand{\three}{\ensuremath{(\mathrm{iii})}}

\newcommand{\CC}{\ensuremath{\mathbb{C}}} 
\newcommand{\kk}{\ensuremath{\Bbbk}} 
\newcommand{\NN}{\ensuremath{\mathbb{N}}} 
\newcommand{\QQ}{\ensuremath{\mathbb{Q}}}

\newcommand{\ZZ}{\ensuremath{\mathbb{Z}}}

\newcommand{\Amp}{\operatorname{Amp}}

\newcommand{\End}{\operatorname{End}}

\newcommand{\git}{\ensuremath{\operatorname{\!/\!\!/\!}}}
\newcommand{\GL}{\operatorname{GL}}
\newcommand{\head}{\operatorname{h}}
\newcommand{\Hilb}{\operatorname{Hilb}}
\newcommand{\Hom}{\operatorname{Hom}}
 
\newcommand{\Irr}{\operatorname{Irr}}

\newcommand{\Mov}{\operatorname{Mov}}

\newcommand{\Nef}{\operatorname{Nef}}

\newcommand{\Pic}{\operatorname{Pic}}
\newcommand{\Proj}{\operatorname{Proj}}
\newcommand{\Quot}{\operatorname{Quot}}
\newcommand{\rank}{\operatorname{rk}}
 
\newcommand{\red}{\operatorname{red}}
\newcommand{\Rep}{\operatorname{Rep}}
\newcommand{\SL}{\operatorname{SL}}

\newcommand{\Sp}{\operatorname{Sp}}
\newcommand{\Spec}{\operatorname{Spec}}
\newcommand{\Sym}{\operatorname{Sym}}

\newcommand{\tail}{\operatorname{t}}

\makeatletter
\@namedef{subjclassname@2020}{%
  \textup{2020} Mathematics Subject Classification}
\makeatother

\title[Hilbert schemes of ADE singularities]{An introduction to Hilbert schemes of points \\ on ADE singularities}

\author{Alastair Craw} 
\address{Department of Mathematical Sciences, 
University of Bath, 
Claverton Down, 
Bath BA2 7AY, 
UK.}
\email{a.craw@bath.ac.uk}
\urladdr{http://people.bath.ac.uk/ac886/}


\begin{document}

\begin{abstract}
This paper is based on a talk at the conference \emph{The McKay correspondence, mutation and related topics} from July 2020. We provide an introduction to joint work of the author with S{\o}ren Gammelgaard, \'{A}d\'{a}m Gyenge and Bal\'{a}zs Szendr\H{o}i~\cite{CGGS19} that constructs the reduced scheme underlying the Hilbert scheme of $n$ points on an ADE singularity as a Nakajima quiver variety for a particular stability parameter. After drawing a parallel with two well-known constructions of the Hilbert scheme of $n$ points in $\mathbb{A}^2$, we summarise results of the author and Gwyn Bellamy~\cite{BC20} before describing the main result by cornering a noncommutative algebra obtained from the preprojective algebra of the framed McKay graph. 
\end{abstract}

\maketitle
\tableofcontents

\section{Introduction}
The study of moduli spaces is of fundamental importance in algebraic geometry and geometric representation theory. 
One much-studied moduli space is the Hilbert scheme of $n$ points in the affine plane for $n\geq 1$. 
Many geometric properties of $\Hilb^{[n]}(\mathbb{A}^2)$ 
can be deduced from a beautiful result of Fogarty~\cite{Fogarty68} describing 
the Hilbert scheme of $n$ points on any smooth surface as a projective crepant resolution of the $n$th symmetric product of the surface; in particular, $\Hilb^{[n]}(\mathbb{A}^2)$ is smooth, it's projective over the affine variety $\Sym^n(\mathbb{A}^2)$, and its canonical bundle is trivial. In addition, $\Hilb^{[n]}(\mathbb{A}^2)$ enjoys the distinction of being a Nakajima quiver variety and, as such, it is a holomorphic symplectic variety that can be obtained as a Geometric Invariant Theory (GIT) quotient. Our aim here is to describe the extent to which these statements about $\Hilb^{[n]}(\mathbb{A}^2)$ generalise to the Hilbert scheme of $n$ points on any ADE singularity.

 By an ADE singularity, we mean any singular surface of the form $\mathbb{A}^2/\Gamma$ for some finite subgroup $\Gamma\subset \SL(2,\kk)$, where $\kk$ is an algebraically closed field of characteristic zero. These surfaces (also known as Kleinian, Du Val, or simple surface singularities) are well known to be hypersurfaces in $\mathbb{A}^3$ with an isolated singularity at the origin, and their unique minimal resolution $S\to \mathbb{A}^2/\Gamma$ contracts a tree of rational curves in a configuration encoded by an ADE Dynkin diagram. Moreover, the many guises of the McKay correspondence~\cite{McKay80, GSV83,Reid02} describe the geometry of $S$ in terms of the $\Gamma$-equivariant geometry of $\mathbb{A}^2$ or, equivalently, in terms of the representation theory of $\Gamma$. In fact, so much is known about ADE singularities that it came as a surprise (at least to the author) that very little was known about $\Hilb^{[n]}(\mathbb{A}^2/\Gamma)$ until fairly recently. 
 
 It turns out that $\Hilb^{[n]}(\mathbb{A}^2/\Gamma)$ shares 
 some key geometric properties with the Hilbert scheme of $n$ points on $\mathbb{A}^2$, at least if we give it 
 the underlying reduced scheme structure.
 
 
 \begin{theorem}
 \label{thm:mainintro}
  Let $\Gamma\subset \SL(2,\kk)$ be a finite subgroup and let $n\geq 1$. Then  $\Hilb^{[n]}(\mathbb{A}^2/\Gamma)_{\red}\!:$
  \begin{enumerate}
      \item[\one] is a Nakajima quiver variety for some non-generic GIT stability parameter;
      \item[\two] is a partial crepant resolution of the affine variety $\Sym^n(\mathbb{A}^2/\Gamma);$ and
      \item[\three] admits a unique projective crepant resolution obtained by variation of GIT quotient.
  \end{enumerate}
\end{theorem}

 It follows, for example, that $\Hilb^{[n]}(\mathbb{A}^2/\Gamma)$ is irreducible (a result due to Zheng~\cite{Zheng17}), and its underlying reduced scheme is normal and has symplectic singularities. This implies in particular that its singularities are both rational and Gorenstein. The fact that $\Hilb^{[n]}(\mathbb{A}^2/\Gamma)$ admits a unique crepant resolution was known in the special case $n=2$ by the work of Yamagishi~\cite{Yamagishi17}.

 Our goal in writing this expository paper is to motivate, to describe, and then to sketch the proof of Theorem~\ref{thm:mainintro} that forms the main result of the author's work with Gammelgaard, Gyenge and Szendr\H{o}i~\cite{CGGS19} and which builds on the author's earlier work with Bellamy~\cite{BC20}. Following the 
 remit from the editors of this conference proceedings, we have placed significant emphasis while writing this paper on the motivation of this result and, for that reason, much of our initial effort focuses on the analogous statements for the Hilbert scheme of points on $\mathbb{A}^2$. This material is classical, and there are a number of excellent summaries in the literature that delve much deeper than we do here, including those by Nakajima~\cite{Nakajimabook}, Bertin~\cite{Bertin12} and
Ginzburg~\cite{Ginzburg12}, as well as the unpublished notes by Bolognese--Losev~\cite{BL14}. Nevertheless, we have chosen to review carefully these results for $\Hilb^{[n]}(\mathbb{A}^2)$ because it provides the opportunity to introduce our point-of-view that makes the statements for $\Hilb^{[n]}(\mathbb{A}^2/\Gamma)$ listed in Theorem~\ref{thm:mainintro} seem especially natural.

 
 \smallskip
 
 It is worth explaining briefly what our point-of-view is here.
 There is a well known quiver GIT description of $\Hilb^{[n]}(\mathbb{A}^2)$ that on the face of it has nothing to do with Nakajima quiver varieties, simply because the relevant quiver (see Figure~\ref{fig:quiverQ'}) has only three arrows and therefore cannot be the doubled quiver of any graph. The key insight, going back to Nakajima~\cite[Proposition~2.7]{Nakajimabook}, is that if one augments the quiver by introducing an extra arrow and uses the preprojective relation on the resulting quiver, then the original GIT description can be reinterpreted as a Nakajima quiver variety for a particular choice of GIT stability condition. 
 This enables one to regard some classical geometric results for $\Hilb^{[n]}(\mathbb{A}^2)$ as corollaries of the Nakajima quiver variety construction. For example, Fogarty's description~\cite{Fogarty68} of the Hilbert--Chow morphism
 \[
 \Hilb^{[n]}\big(\mathbb{A}^2\big)\longrightarrow \Sym^n\big(\mathbb{A}^2\big)
 \]
 as a projective, crepant resolution of singularities can be obtained directly as a statement about variation of GIT quotient for Nakajima quiver varieties (see Theorem~\ref{thm:Fogarty}).
 
 An analogous dual GIT interpretation of the Hilbert scheme of points on $\mathbb{A}^2/\Gamma$ lies at the heart of Theorem~\ref{thm:mainintro}. Just as for $\Hilb^{[n]}(\mathbb{A}^2)$, it is relatively straightforward to introduce a natural quiver GIT construction of $\Hilb^{[n]}(\mathbb{A}^2/\Gamma)$ in terms of a simple quiver (see Figure~\ref{fig:quiverQ3loops}) and an ideal of relations involving the defining hypersurface equation $(f=0)\subset \mathbb{A}^3$ of the ADE singularity $\mathbb{A}^2/\Gamma$. 
 However, 
 the second GIT construction in this situation requires more than a simple reinterpretation.  For this, we study Nakajima quiver varieties associated to the affine Dynkin graph of type ADE  associated to $\Gamma$ by McKay~\cite{McKay80}. Fix dimension vectors $\mathbf{v}= n\delta$ and $\mathbf{w}=\rho_0$, where $\delta$ (resp.\ $\rho_0$) corresponds to the regular (resp.\ trivial) representation of $\Gamma$. The affine Nakajima quiver variety $\mathfrak{M}_0(\mathbf{v},\mathbf{w})$ is isomorphic to $\Sym^n(\mathbb{A}^2/\Gamma)$, while the Nakajima quiver variety $\mathfrak{M}_\theta(\mathbf{v},\mathbf{w})$ determined by any stability condition $\theta$ in a particular open GIT chamber $C_+$ is isomorphic as a variety to
 \[
 n\Gamma\text{-}\Hilb(\mathbb{A}^2) = \big\{ \Gamma\text{-invariant ideals }I\subset \kk[x,y] \mid \kk[x,y]/I\cong_\Gamma \kk[\Gamma]^{\oplus n}\big\}
 \]
 by \cite{VV99, Wang02}. Our approach to Theorem~\ref{thm:mainintro} is to vary the stability parameter in two stages: first, from the GIT chamber $C_+$ to a specific parameter $\theta \rightsquigarrow\theta_0$ in the boundary of $C_+$; and then to the origin $\theta_0\rightsquigarrow 0$, thereby inducing morphisms 
\begin{equation}
    \label{eqn:morphsims}
\begin{tikzcd}
n\Gamma\text{-}\Hilb(\mathbb{A}^2)\cong \mathfrak{M}_{\theta}(\mathbf{v},\mathbf{w}) \ar[r,"\ref{thm:mainintro}\three"]  &  
\mathfrak{M}_{\theta_0}(\mathbf{v},\mathbf{w})\ar[r,"\ref{thm:mainintro}\two"]  &  \mathfrak{M}_{0}(\mathbf{v},\mathbf{w})\cong \Sym^n(\mathbb{A}^2/\Gamma);
 \end{tikzcd}
\end{equation}
 here, the right-hand morphism will provide the partial crepant resolution from the statement of Theorem~\ref{thm:mainintro}\two, while the left-hand morphism will provide the projective, symplectic resolution appearing in Theorem~\ref{thm:mainintro}\three. All that remains, then, is to establish the isomorphism
\begin{equation}
\label{eqn:mainone}
\mathfrak{M}_{\theta_0}(\mathbf{v},\mathbf{w})\cong \Hilb^{[n]}(\mathbb{A}^2/\Gamma)_{\textrm{red}}
\end{equation}
required for Theorem~\ref{thm:mainintro}\one. 

Constructing this isomorphism forms the heart of the argument in \cite{CGGS19}. The first step is to show 
that a specific globally generated line bundle $L$ on $\mathfrak{M}_{\theta}(\mathbf{v},\mathbf{w})$ can be obtained in two ways:
\begin{enumerate}
    \item first, as the pullback of the polarising ample line bundle on $\mathfrak{M}_{\theta_0}(\mathbf{v},\mathbf{w})$ via the morphism $\mathfrak{M}_{\theta}(\mathbf{v},\mathbf{w}) \to \mathfrak{M}_{\theta_0}(\mathbf{v},\mathbf{w})$ induced by variation of GIT quotient $\theta\rightsquigarrow \theta_0$; and also
    \item as the pullback via the universal morphism $\mathfrak{M}_{\theta}(\mathbf{v},\mathbf{w}) \to \Hilb^{[n]}(\mathbb{A}^2/\Gamma)$ of the polarising ample bundle obtained from the first GIT construction of $\Hilb^{[n]}(\mathbb{A}^2/\Gamma)$  described above.
\end{enumerate}
 To obtain the universal morphism, we construct an algebra $A_0$ from the preprojective algebra $\Pi$ of the framed McKay quiver by deleting one arrow and then \emph{cornering} away all bar two of the vertices. The fact that the same line bundle pulls back via both morphisms, the first of which is surjective, allows us to construct a closed immersion $\mathfrak{M}_{\theta_0}(\mathbf{v},\mathbf{w})\hookrightarrow \Hilb^{[n]}(\mathbb{A}^2/\Gamma)$. The second step of the proof is more algebraic: we study a functor arising in a recollement of a module category in order to deduce that the above closed immersion induces the isomorphism \eqref{eqn:mainone}. We conclude by describing the generalisation that constructs orbifold Quot schemes as quiver varieties~\cite{CGGS21}.

\subsection*{Acknowledgements}
Thanks to S\o ren, \'{A}d\'{a}m and Bal\'{a}zs for their efforts and inspiration on our joint project \cite{CGGS19,CGGS21}. I owe a significant debt to Gwyn for reigniting my interest in Nakajima quiver varieties and from whom I have learned a great deal since we began work on \cite{BC20}. I'm grateful to Akira Ishii, Yukari Ito and Osamu Iyama for running the \emph{The McKay correspondence, mutation and related topics} conference online in July--August 2020; 
I owe special thanks to Akira for introducing me to the linearisation map almost two decades ago.
Thanks to Ben McKay for writing the excellent Dynkin diagrams package with which I drew Figure~\ref{fig:Dynkin}.
Finally, thanks to the anonymous referee for several very helpful comments and questions.



\section{How to calculate \texorpdfstring{$\Hilb^{[n]}(\mathbb{A}^2)$}{affine plane}}

\label{sec:HilbnA2} 
Throughout we let $\kk$ denote an algebraically closed field of characteristic zero. 

\subsection{As a set}
For $n\geq 1$, the \emph{Hilbert scheme of $n$ points in $\mathbb{A}^2$} parametrises subschemes of length $n$ in the affine plane, or equivalently, ideals in $\kk[x,y]$ for which the quotient has dimension $n$ over $\kk$:
\begin{eqnarray*}
\Hilb^{[n]}(\mathbb{A}^2) & := & \big\{Z\subset \mathbb{A}^2 \mid \dim_{\kk} H^0(\mathcal{O}_Z) = n\big\} \\
& = & \big\{ I\subset \kk[x,y] \mid \dim_{\kk} \kk[x,y]/I = n\big\}.
 \end{eqnarray*}
 For example, any collection of $n$ distinct closed points in $\mathbb{A}^2$ defines a point of $\Hilb^{[n]}(\mathbb{A}^2)$, as does the scheme  supported at the origin in $\mathbb{A}^2$ cut out by the ideal $\langle x^n,y\rangle$.
 
 \subsection{As a GIT quotient}
 If we fix a vector space $V_0$ of dimension $n$ with a pair of commuting endomorphisms $B_1, B_2\in \End(V_0)$, then the resulting $\kk[x,y]$-module $V_0$ is isomorphic to the space of sections $H^0(\mathcal{O}_Z)=\kk[x,y]/I$ for some subscheme $Z\subset \mathbb{A}^2$ of length $n$ if and only if there is a vector $i\in V_0$ that generates $V_0$ as a $\kk[x,y]$-module. 
 
 To see how this observation allows us to construct $\Hilb^{[n]}(\mathbb{A}^2)$ as a GIT quotient, choose a basis of $V_0$ and define 
 \begin{equation}
     \label{eqn:Z}
 Z:=\big\{(B_1, B_2, i)\in \End(\kk^n)^{\oplus 2} \oplus \kk^n \mid B_1B_2=B_2B_1
\big\} 
  \end{equation}
 The group $\GL(n,\kk)$ acts on $\kk^n$ by change of basis and the induced action on $Z$ satisfies   
 \[
 g\cdot (B_1,B_2,i) = \big(gB_1g^{-1},gB_2g^{-1},gi\big)
 \]
 for $g\in\GL(n,\kk)$. It turns out that one can choose a stability condition that characterises the set of $\GL(n,\kk)$-orbits of points $(B_1,B_2,i)\in Z$ for which the corresponding $\kk[x,y]$-module $V_0$ is generated by the vector $i\in V_0\cong \kk^n$. For this, consider a character $\chi\in\GL(n,\kk)^\vee$. We say that a polynomial function $f\in \kk[Z]$ is \emph{$\chi$-semi-invariant} if $f(g\cdot z) = \chi(g) f(z)$ for all $z\in Z$ and $g\in \GL(n,\kk)$, and we write $\kk[Z]_{\chi}$ for the vector space of all $\chi$-semi-invariant polynomial functions on $Z$. Create a graded ring by taking the direct sum of these graded pieces over all positive multiples of the character $\chi$. We're interested in the space obtained as
$\Proj$ of this graded ring, namely
 \begin{equation}
      \label{eqn:Proj}
 Z\git_{\chi} \GL(n,\kk):= 
\Proj \bigoplus_{k\geq 0} \kk[Z]_{k\chi}.
 \end{equation}
 This is the \emph{Geometric Invariant Theory (GIT) quotient} of $Z$ by $\GL(n,\kk)$, linearised by $\chi$.
  
 \begin{theorem}[\textbf{GIT construction of $\Hilb^{[n]}(\mathbb{A}^2)$}]
 \label{thm:HilbnGIT}
  Let $\chi\in \GL(n,\kk)^\vee$ be such that $\chi(g) = \det(g)$ for all $g\in \GL(n,\kk)$. Then $\Hilb^{[n]}(\mathbb{A}^2)$ can be constructed as the GIT quotient $Z\git_{\chi} \GL(n,\kk)$.
 \end{theorem}
 
 \begin{remark}
 \label{rem:stability}
  To explain the notation, note that the irrelevant ideal for the $\Proj$ construction from \eqref{eqn:Proj} is $\bigoplus_{k> 0} \kk[Z]_{k\chi}$. Removing from $Z$ the locus where all functions in this ideal vanish leaves the open set of \emph{$\chi$-semistable points}; explicitly, $z\in Z$ is \emph{$\chi$-semistable} if there exists $k> 0$ and $f\in \kk[Z]_{k\chi}$ such that $f(z)\neq 0$. If, in addition, the stabiliser $\GL(n,\kk)_z$ is finite and all $\GL(n,\kk)$-orbits are closed in $Z\setminus (f=0)$, then $z$ is said to be \emph{$\chi$-stable}. We'll see in the proof of Theorem~\ref{thm:HilbnGIT} that for the character $\chi$ from Theorem~\ref{thm:HilbnGIT} (indeed, for any nonzero character), every $\chi$-semistable point of $Z$ is $\chi$-stable, in which case it follows that  $Z\git_{\chi} \GL(n,\kk)$
 parametrises the $\GL(n,\kk)$-orbits of all $\chi$-stable points in $Z$. In this sense, then, the GIT quotient $Z\git_{\chi} \GL(n,\kk)$ is indeed a space of $\GL(n,\kk)$-orbits.
  \end{remark}
 
  \subsection{Proof via quiver GIT}
  To prove Theorem~\ref{thm:HilbnGIT}, it remains to show that $(B_1, B_2,i)\in Z$ is $\chi$-stable if and only if the vector $i\in \kk^n$ generates the corresponding
  $\kk[x,y]$-module $V_0\cong \kk^n$. 
  
  For this we recall a notion of stability phrased directly in terms of the $\kk[x,y]$-module $V_0$. Consider the quiver $\mathcal{Q}$ with vertex set $\mathcal{Q}_0=\{\infty,0\}$ and arrow set $\mathcal{Q}_1=\{a_1, a_2, b\}$, where $a_1, a_2$ are loops at vertex $0$ and where the tail and head of the arrow $b$ lie at $\infty$ and $0$ respectively as shown in Figure~\ref{fig:quiverQ'}.
 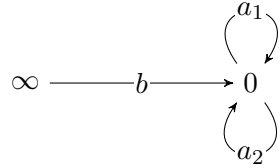
\begin{figure}[h!]
\centering
\begin{tikzpicture}[yscale=0.77]
\node (M0) at (-1,0) {$\infty$};
\node (M1) at (2,0) {$0$};
\draw [->] (M0) to node[gap] {$b$} (M1);
\draw [->,looseness=12, out=120, in=60] (M1) to node[gap] {$\small{a_1}$} (M1);
\draw [->,looseness=12, out=-60, in=-120] (M1) to node[gap] {$\small{a_2}$} (M1);
\end{tikzpicture}
\caption{The quiver $\mathcal{Q}$ with the relation $a_1a_2-a_2a_1$.}
\label{fig:quiverQ'}
\end{figure}
  By associating the vector spaces $V_\infty=\kk$ and $V_0=\kk^n$ to the vertices $\infty$ and $0$ respectively, and linear maps $i\in\Hom(\kk,\kk^n)$ and $B_1, B_2\in \End(\kk^n)$ to the arrows $b$ and $a_1, a_2$ respectively, we obtain a \emph{representation of the quiver} $\mathcal{Q}$ of dimension vector $(1,n)$. Imposing the \emph{relation} $a_1a_2-a_2a_1$ on the arrows ensures that the endomorphisms $B_1, B_2$ commute. It follows that representations satisfying this relation correspond precisely to $\kk[x,y]$-modules $V_0$ of dimension $n$ over $\kk$ together with the choice of a vector $i\in V_0$ defined to be the image of $1$ under the map $i\colon \kk\to \kk^n$. Put simply, every such representation is equivalent to the choice of a triple $(B_1, B_2, i)\in Z$. Moreover, isomorphism classes of representations correspond to $\GL(n,\kk)$-orbits of the triple; indeed, both notions simply encode a change of basis on $V_0$. We note in passing that the category of representations of $\mathcal{Q}$ satisfying the relation is equivalent to the category of left-modules over the noncommutative algebra \[
\Lambda:=\kk \mathcal{Q}/(a_1a_2-a_2a_1),
\]
where $\kk \mathcal{Q}$ denotes the path algebra of the quiver $\mathcal{Q}$.

To translate the GIT notion of stability from Remark~\ref{rem:stability} into one for representations of $\mathcal{Q}$, we fix an  isomorphism $\ZZ\cong \GL(n,\kk)^\vee$. There are two candidates for this isomorphism, and different authors make different choices; here, we choose the isomorphism sending $\theta\in \ZZ$ to the character
\begin{equation}
\label{eqn:character}
\chi_\theta\colon \GL(n,\kk)\to \mathbb{C}^\times ,\quad g\mapsto \det(g)^\theta.
\end{equation}
 Identify $\theta\in \ZZ$ with the map $\theta\colon \ZZ^2\to \ZZ$ given by  $\theta(v_\infty,v_0) = -n\theta v_\infty + \theta v_0$. For any subrepresentation  $V^\prime\subseteq V$ determined by subspaces $V^\prime_\infty\subseteq V_\infty=\kk$ and $V^\prime_0\subseteq V_0=\kk^n$ at the vertices of $\mathcal{Q}$, we write 
 \[
\theta(V^\prime):= \theta\big(\!\dim_\kk V^\prime_\infty,\dim_\kk V^\prime_0\big) = -n \theta\dim_\kk V^\prime_\infty + \theta\dim_\kk V^\prime_0\in \ZZ.
\]

  \begin{definition}[\textbf{King stability}]
  \label{def:stability}
  A representation $V$ of $\mathcal{Q}$ of dimension $(1,n)$ is \emph{$\theta$-semistable} (resp.\ \emph{$\theta$-stable}) if every subrepresentation $0\subsetneq V'\subsetneq V$ satisfies $\theta(V^\prime)\geq 0$ (resp.\ $\theta(V^\prime)>0$).
    \end{definition}
  
  We follow the sign convention of King~\cite{King94}, but note that some authors replace the exponent $\theta$ in \eqref{eqn:character} by $-\theta$ and the inequalities from Definition~\ref{def:stability} by $\theta(V^\prime)\leq 0$ (and $\theta(V^\prime) < 0$). With either convention, 
  \cite[Theorem~2.2]{King94} shows that $(i,B_1,B_2)\in Z$ is $\chi_\theta$-(semi)stable in the sense of  Remark~\ref{rem:stability} if and only if the corresponding representation $V=V_\infty\oplus V_0$ of $\mathcal{Q}$ is $\theta$-(semi)stable. 
 
 \begin{proof}[Proof of Theorem~\ref{thm:HilbnGIT}]
    For the character $\chi$ from Theorem~\ref{thm:HilbnGIT}, we show that a point $(B_1, B_2,i)\in Z$ is $\chi$-stable if and only if the corresponding $\kk[x,y]$-module $V_0$ is generated by the vector $i\in V_0$. King's result shows that $(B_1, B_2,i)\in Z$ is $\chi$-stable if and only if $V=V_\infty\oplus V_0$ is $\theta$-stable as a representation of $\mathcal{Q}$; here, the isomorphism \eqref{eqn:character} gives $\chi=\chi_\theta$ for $\theta=1$. For any subrepresentation $0\subsetneq V'\subsetneq V$, we have $ \theta(V^\prime)= -n\dim_\kk V^\prime_\infty + \dim_\kk V^\prime_0$. If $V_\infty^\prime=0$ then  $\theta(V^\prime)>0$. It follows that $V$ is $\theta$-stable (in fact, $\theta$-semistable) if and only if any such $V^\prime$ with $V_\infty^\prime\neq 0$ satisfies $V^\prime=V$ which, in turn, holds if and only if $V_0$ is generated as a $\kk[x,y]$-module by the vector $i$ in the image of the map $V_\infty\to V_0$. This concludes the construction of $\Hilb^{[n]}(\mathbb{A}^2)$ as the GIT quotient \eqref{eqn:Proj}.
  \end{proof}
 
 Theorem~\ref{thm:HilbnGIT} provides in particular a simple construction of $\Hilb^{[n]}(\mathbb{A}^2)$ as an algebraic scheme. 
 
 \begin{remark}
 \label{rem:fine1}
 For $\theta=1$, Theorem~\ref{thm:HilbnGIT} implies that $\Hilb^{[n]}(\mathbb{A}^2)$ may be regarded as the \emph{fine moduli space} of $\theta$-stable $\Lambda$-modules 
 of dimension $(1,n)$;  the phrase `$\infty$-generated' is often used in place of `$\theta$-stable' for this (positive) choice of $\theta$. This means that  $\Hilb^{[n]}(\mathbb{A}^2)$ carries:
 \begin{enumerate}
     \item[\one] a tautological vector bundle $\mathcal{O}_{\Hilb}\oplus \mathcal{T}$ whose fibre $\kk\oplus H^0(\mathcal{O}_Z)$ over any closed point $[Z]$ is the $\theta$-stable $\Lambda$-module $V$ appearing in the proof of Theorem~\ref{thm:HilbnGIT}; and
 \item[\two] a tautological $\kk$-algebra homomorphism from the algebra $\Lambda$ to $\End(\mathcal{O}_{\Hilb}\oplus \mathcal{T})$ that encodes the maps between summands of $\mathcal{O}_{\Hilb}\oplus \mathcal{T}$ determined by arrows in the quiver $\mathcal{Q}$.
 \end{enumerate}
 Moreover, $\Hilb^{[n]}(\mathbb{A}^2)$ is universal with this property, in the sense that any flat family of $\theta$-stable $\Lambda$-modules of dimension vector $(1,n)$ over a scheme $X$ is obtained as the pullback of $\mathcal{O}_{\Hilb}\oplus \mathcal{T}$ via a unique morphism $X\to \Hilb^{[n]}(\mathbb{A}^2)$.
 \end{remark}

 \subsection{As a Nakajima quiver variety}
  We now construct $\Hilb^{[n]}(\mathbb{A}^2)$ as a quiver variety in the sense of Nakajima~\cite{Nakajima94,Nakajima98}. This also provides the opportunity to recall a beautiful result of Fogarty~\cite{Fogarty68} that records several key geometric properties of $\Hilb^{[n]}(\mathbb{A}^2)$.
 
 First, define $\Sym^n(\mathbb{A}^2)$ to be the affine variety obtained as the quotient of $\mathbb{A}^{2n}=\prod_{i=1}^n \Spec \kk[x_i,y_i]$ by the action of the symmetric group $\mathfrak{S}_n$ that permutes the factors. This action preserves the symplectic form $\sum_{1\leq i\leq n} dx_i\wedge dy_i$ on $\mathbb{A}^{2n}$, so in fact $\mathfrak{S}_n$ is a subgroup of $\Sp(2n,\kk)$. This symplectic form descends to a non-degenerate closed 2-form $\omega_{\text{reg}}$ on the smooth locus of the affine variety
 \[
 \Sym^n(\mathbb{A}^2):= \mathbb{A}^{2n}_\kk/\mathfrak{S}_n = \Spec \big(\kk[x_1,y_1,\dots,x_{n},y_{n}]^{\mathfrak{S}_n}\big).
  \]
 It turns out that the pullback of $\omega_{\text{reg}}$ via any resolution of singularities $\varphi\colon X \to \Sym^n(\mathbb{A}^2)$ can be extended to a regular 2-form on the whole of $X$; following Beauville~\cite[Proposition~2.4]{Beauville00}, we say that $\Sym^n(\mathbb{A}^2)$ has \emph{symplectic singularities}.
 
 A variety $Y$ with symplectic singularities is said to admit a \emph{symplectic resolution} $\varphi\colon X \to Y$ if the pullback of $\omega_{\text{reg}}$ extends to a non-degenerate regular 2-form on $X$. This makes $X$ a smooth \emph{holomorphic symplectic} variety. It turns out that $\Sym^n(\mathbb{A}^2)$ admits a (unique) projective symplectic resolution of singularities as follows.
 
 \begin{theorem}[\textbf{Geometric properties of $\Hilb^{[n]}(\mathbb{A}^2)$}]
 \label{thm:fogarty}
 There is a projective, symplectic resolution
  \begin{equation}
     \label{eqn:HilbChow}
 \Hilb^{[n]}(\mathbb{A}^2)\longrightarrow \Sym^n(\mathbb{A}^2)=\mathbb{A}^{2n}_\kk/\mathfrak{S}_n.
 \end{equation}
In particular, $\Hilb^{[n]}(\mathbb{A}^2)$ is a smooth, holomorphic symplectic variety of dimension $2n$ that is projective over the Gorenstein affine singularity $\Sym^n(\mathbb{A}^2)$.
 \end{theorem}

 We prove this result by tweaking slightly the GIT construction from Theorem~\ref{thm:HilbnGIT}, leading to the construction of $\Hilb^{[n]}(\mathbb{A}^2)$ as a \emph{Nakajima quiver variety}. We postpone a more general discussion of quiver varieties, choosing for now to describe only the straightforward construction of $\Hilb^{[n]}(\mathbb{A}^2)$. 
 
 Our starting data is a graph and a pair of dimension vectors $\mathbf{v}, \mathbf{w}$: in the case of interest, the graph comprises a single vertex $0$ with one loop, and the dimension vectors are $\mathbf{v}:=n, \mathbf{w}:=1\in \ZZ$. Since $\mathbf{w}=1$, we augment the graph by adding a new \emph{framing vertex}, denoted $\infty$, together with a single edge joining $\infty$ to $0$. The doubled quiver $Q$ associated to this graph shares the same vertex set and replaces each edge by a pair of arrows,  one pointing in each direction (see Figure~\ref{fig:nakajima}). 
  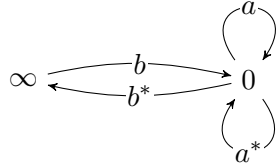
\begin{figure}[ht!]
\centering
\begin{tikzpicture}[yscale=0.77]
\node (M0) at (-1,0) {$\infty$};
\node (M1) at (2,0) {$0$};
\draw[->, out=15, in=165] (M0) to node[gap] {$\small{b}$} (M1);
\draw[->, out=195, in=345] (M1) to node[gap] {$\small{b^*}$} (M0);
\draw [->,looseness=12, out=120, in=60] (M1) to node[gap] {$\small{a}$} (M1);
\draw [->,looseness=12, out=-60, in=-120] (M1) to node[gap] {$\small{a^*}$} (M1);
\end{tikzpicture}
\caption{The quiver $Q$ with relation $aa^*-a^*a+bb^*=0$.}
\label{fig:nakajima}
\end{figure}
 Consider the relation $aa^*-a^*a+bb^*=0$ on $Q$. The category of representations of $Q$ satisfying this relation is equivalent to the category of left modules over the \emph{framed preprojective algebra} 
\[
\Pi:= \kk Q/\big(aa^*-a^*a+bb^*\big).
\]
 We choose not to introduce $b^*b$ as a relation here, but compare Remarks~\ref{rems:preprojrelations}. 
 
  To extract the geometry, we mimic the description from the previous section using quiver representations. That is, a representation of the quiver $Q$ of dimension vector $(1,n)$ satisfying the relation is a quadruple $(B_a,B_{a^*},i,j)\in \mu^{-1}(0)$, where $\mu$ is the map
 \begin{equation}
     \label{eqn:mu}
 \mu\colon \End(\kk^n)^{\oplus 2} \oplus \kk^n \oplus (\kk^n)^*\longrightarrow \End(\kk^n) \; : \; (B_a, B_{a^*}, i,j)\mapsto [B_a,B_{a^*}] +ij.
 \end{equation}
 Consider the action of $\GL(n,\kk)$ on the locus $\mu^{-1}(0)$, where now $g\in\GL(n,\kk)$ acts by 
 \[
 g\cdot (B_a,B_{a^*},i,j) = \big(gB_ag^{-1},gB_{a^*}g^{-1},gi,jg^{-1}\big).
 \]
 Stability conditions for representations of $Q$ are maps $\theta\colon \ZZ^2\to \ZZ$ given by  $\theta(v_\infty,v_0) = -n\theta v_\infty + \theta v_0$
 for some $\theta\in \ZZ$, each corresponding to the character $\chi_\theta\in \GL(n,\kk)^\vee$ from \eqref{eqn:character}. 
 
 \begin{definition}[\textbf{Nakajima quiver variety for} $Q$] 
 \label{def:nakajimaGIT}
 For $\theta\in \ZZ$ and the combinatorial data as above, the associated \emph{Nakajima quiver variety} is defined to be
 the GIT quotient
 \[
 \mathfrak{M}_{\theta}(\mathbf{v},\mathbf{w})= \mu^{-1}(0)\git_{\chi_\theta} \GL(n,\kk):= 
\Proj \bigoplus_{k\geq 0} \kk\big[\mu^{-1}(0)\big]_{k\chi_\theta}.
 \]
 \end{definition}

The special case $\theta=0$ defines the affine GIT quotient which satisfies 
 \[
 \mathfrak{M}_{0}(\mathbf{v},\mathbf{w})\cong
\Spec \Big(\kk\big[\mu^{-1}(0)\big]^{\GL(n,\kk)}\Big).
\]
Every representation of $Q$ of dimension vector $(1,n)$ is necessarily semistable with respect to the zero stability condition $0\in \ZZ^2$, and the inclusion of the $\chi_\theta$-stable locus $\mu^{-1}(0)^{\chi_\theta\text{-s}}$ into $\mu^{-1}(0)$ for any $\theta\in \ZZ$ is a $\GL(n,\kk)$-equivariant map that gives rise to a morphism $f_\theta\colon \mathfrak{M}_{\theta}(\mathbf{v},\mathbf{w})\rightarrow \mathfrak{M}_{0}(\mathbf{v},\mathbf{w})$ that is said to be obtained by \emph{variation of GIT quotient} (VGIT). The following result implies in particular the result of Fogarty~\cite{Fogarty68} for $\Hilb^{[n]}(\mathbb{A}^2)$ mentioned above.
 
 
 \begin{theorem}[\textbf{$\Hilb^{[n]}(\mathbb{A}^2)$ as a quiver variety}]
 \label{thm:Fogarty}
 For $\theta=1$, there is a commutative diagram
\begin{equation}
    \label{eqn:diagFogarty}
 \begin{tikzcd}
\mathfrak{M}_{1}(\mathbf{v},\mathbf{w})\ar[r, "f_1"] \ar[d,swap,"\sim"] & \mathfrak{M}_0(\mathbf{v},\mathbf{w}) \ar[d,"\sim"]   \\
\Hilb^{[n]}(\mathbb{A}^2) \ar[r]  &  \Sym^n(\mathbb{A}^2)
 \end{tikzcd}
\end{equation}
  of schemes over $\kk$ in which the vertical maps are isomorphisms and the horizontal maps are projective symplectic resolutions. 
In particular, the statement of Theorem~\ref{thm:fogarty} holds.
 \end{theorem}
 
  As noted prior to Theorem~\ref{thm:fogarty}, the morphism $f_1$ is the unique projective symplectic resolution of $\Sym^n(\mathbb{A}^2)$. For more on this, see Remark~\ref{rem:uniqueness}.

  To prove this result,  compare the GIT description of $\Hilb^{[n]}(\mathbb{A}^2)$ from Theorem~\ref{thm:HilbnGIT} with that of $\mathfrak{M}_{1}(\mathbf{v},\mathbf{w})$ from Definition~\ref{def:nakajimaGIT}: representations of the quiver from Figure~\ref{fig:nakajima} differ from those of the quiver in Figure~\ref{fig:quiverQ'} only in the existence of the map $j$ corresponding to the arrow $b^*$. The following lemma, which is a special case of \cite[Lemma~2.8]{Nakajima98} (see also Ginzburg~\cite[Proposition~5.6.5]{Ginzburg12}),
  shows that for the representations of interest, this map is necessarily the zero map.
  
  \begin{lemma}
  \label{lem:j=0}
  Let $(B_a,B_{a^*},i,j)\in \mu^{-1}(0)$ be $\chi_1$-stable.  Then $j=0$.
  \end{lemma}
  \begin{proof}
  Since $ij$ factors through $V_\infty=\kk$, the equality $[B_a,B_{a^*}]=-ij$ implies that $[B_a,B_{a^*}]$ has rank at most one. After a change of basis on $\kk^n$, $B_a$ and $B_{a^*}$ can simultaneously be put in upper triangular form (see~\cite[Lemma~12.7]{EG02} for an elegant proof of this linear algebra statement). For any $M\in \End(\kk^n)$, applying the cyclic property of trace and the equation defining $\mu^{-1}(0)$ gives
  \[
  jMi = \mathrm{Tr}(jMi) = \mathrm{Tr}(Mij) =  -\mathrm{Tr}(M[B_a,B_{a^*}]).
  \]
  Now let $M$ be the product of finitely many occurrences of $B_a$ and $B_{a^*}$ taken in any order. Then $M$ is upper triangular and, since $[B_a,B_{a^*}]$ is strictly upper triangular, so is $M[B_a,B_{a^*}]$, and hence $jMi=0$ for every such $M$. The proof of Theorem~\ref{thm:HilbnGIT} applies verbatim to show that each vector in $V_0$ is a linear combination of vectors in the image of maps $Mi\colon V_\infty\to V_0$ for endomorphisms $M$ as above.   Since $jMi=0$ for such $M$, the map $j\colon V_0\to V_\infty$ is zero.
  \end{proof}
 
  \begin{remark}
  \label{rem:mu}
 Before embarking on the proof of Theorem~\ref{thm:Fogarty}, it is worth commenting on the fact that the map $\mu$ from \eqref{eqn:mu} is a \emph{moment map} for the action of $\GL(n,\kk)$ on the space $\End(\kk^n)^{\oplus 2}\oplus \kk^n\oplus (\kk^n)^*$ \cite[Section~4.1]{Ginzburg12}. The key consequence for the next proof is that the locus of points in $\mu^{-1}(0)$ with finite isotropy group is contained in the non-singular locus of $\mu^{-1}(0)$ \cite[Lemma~4.1.7]{Ginzburg12}.
  \end{remark}
  
 \begin{proof}[Proof of Theorem~\ref{thm:Fogarty}]
 We first construct the left-hand isomorphism from \eqref{eqn:diagFogarty}. The inclusion of the locus $Z$ from \eqref{eqn:Z} into $\mu^{-1}(0)$ by sending $(B_a,B_{a^*},i)\mapsto (B_a,B_{a^*},i,0)$ is $\GL(n,\kk)$-equivariant. In light of Lemma~\ref{lem:j=0}, this map identifies the $\chi_1$-stable locus of $Z$ with the $\chi_1$-stable locus of $\mu^{-1}(0)$. The isomorphism $\Hilb^{[n]}(\mathbb{A}^2)\cong \mathfrak{M}_1(\mathbf{v},\mathbf{w})$ follows by comparing the GIT quotient constructions from Theorem~\ref{thm:HilbnGIT} and Definition~\ref{def:nakajimaGIT} (see Section~\ref{sec:fine} for an alternative construction of this isomorphism).
 
 We now show that $\mathfrak{M}_1(\mathbf{v},\mathbf{w})$ is non-singular. Let $(B_a,B_{a^*},i,j)\in \mu^{-1}(0)$ be a $\chi_1$-stable point and let $V$ be the corresponding $\Pi$-module. The proof of Lemma~\ref{lem:j=0} shows that each vector in $V_0$ is a linear combination of vectors in the image of maps $Mi\colon V_\infty\to V_0$, where each $M$ is the product of finitely many occurrences of $B_a$ and $B_{a^*}$ taken in any order. In particular, if $g B_ag^{-1}=B_a$, $gB_{a^*}g^{-1}=B_{a^*}$ and $gi=i$ for some $g\in \GL(n,\kk)$, then $g$ acts trivially on the whole of $V_0=\kk^n$. Thus, the action of $\GL(n,\kk)$ on $\mu^{-1}(0)$ is free at the point $(B_a,B_{a^*},i,j)$, and since this point was arbitrary in the $\chi_1$-stable locus $\mu^{-1}(0)^{\chi_{1}\textrm{-s}}$, it follows that $\GL(n,\kk)$ acts freely on $\mu^{-1}(0)^{\chi_{1}\textrm{-s}}$. Since $\mu$ is the moment map for the $\GL(n,\kk)$-action, the locus $\mu^{-1}(0)^{\chi_{1}\textrm{-s}}$ is non-singular by Remark~\ref{rem:mu}. The group $\GL(n,\kk)$ acts freely on $\mu^{-1}(0)^{\chi_{1}\textrm{-s}}$, so the geometric quotient
 $\mathfrak{M}_1(\mathbf{v},\mathbf{w})$ is non-singular. 

 For the right-hand isomorphism of affine varieties from \eqref{eqn:diagFogarty}, we proceed in two steps. First, since $\mathbf{v}=n$, the decomposition theorem of Crawley--Boevey~\cite[Theorem~1.1]{CB02} gives an isomorphism 
 \[
 \mathfrak{M}_0(\mathbf{v},\mathbf{w})\cong\Sym^n\big( \mathfrak{M}_0(1,\mathbf{w})\big),
 \]
 so it suffices to show that $\mathfrak{M}_0(1,\mathbf{w})\cong \mathbb{A}^2$. For this, the domain of $\mu$ is $\mathbb{A}^4=\Spec \kk[x_a,x_{a^*},x_i,x_j]$, and the preprojective relation cuts out $\mu^{-1}(0)$ as the hypersurface $x_ix_j=0$. The torus $\GL(1,\kk)$ acts trivially on the variables $x_a$ and $x_{a^*}$, and it acts with weights $+1$ and $-1$ on $x_i$ and $x_j$ respectively. Thus, the coordinate ring of $\mathfrak{M}_0(1,\mathbf{w})$ is 
 \[
 \kk\big[\mu^{-1}(0)\big]^{\GL(1,\kk)}\cong \bigg(\frac{\kk[x_a,x_{a^*},x_i,x_j]}{(x_ix_j)}\bigg)^{\GL(1,\kk)}\cong \kk[x_a,x_{a^*}],
 \]
 giving $\mathfrak{M}_0(1,\mathbf{w})\cong \mathbb{A}^2$ as required.

 Finally, the VGIT morphism $f_1$ is well-known to be projective. In addition, the 2-form on $\mu^{-1}(0)$ obtained by restriction from the 2-form $\omega$ satisfying 
 \[
 \omega\big((B_a,B_{a^*},i,j),(B'_a,B'_{a^*},i',j')\big) = \mathrm{Tr}\big( B_aB'_{a^*} - B_{a^*} B_a^\prime\big) + \mathrm{Tr}(ij'-i'j)
 \]
 descends to non-degenerate 2-forms $\omega_1$ and $\omega_0$ on $\mathfrak{M}_1(\mathbf{v},\mathbf{w})$ and the non-singular locus of $\mathfrak{M}_0(\mathbf{v},\mathbf{w})$ respectively. Since $f_1$ is obtained from the $\GL(n,\kk)$-equivariant inclusion of the $\chi_1$-stable locus into $\mu^{-1}(0)$, these 2-forms are compatible in the sense that $\omega_1 \cong f_1^*(\omega_0)$. Thus, $\mathfrak{M}_1(\mathbf{v},\mathbf{w})$ is holomorphic symplectic and $f_1$ is a symplectic resolution of singularities.
\end{proof} 

\begin{remark}
 The heavy lifting in constructing the right-hand isomorphism $\mathfrak{M}_0(\mathbf{v},\mathbf{w})\cong \Sym^n(\mathbb{A}^2)$ in Theorem~\ref{thm:Fogarty} is done by Crawley-Boevey's decomposition theorem. It's worth noting, however, that the symmetric product arises naturally from our quiver-theoretic description of $\Hilb^{[n]}(\mathbb{A}^2)$. Indeed, as in the proof of Lemma~\ref{lem:j=0}, the matrices $B_a, B_{a^*}$ can simultaneously be put in upper triangular form, and 
 the set of $n$ pairs of eigenvalues $(\lambda_k, \nu_k)_{1\leq k\leq n}$ of the matrices $B_a, B_{a^*}$, well-defined up to reordering, defines a point in $\Sym^n (\mathbb{A}^2)$. This provides a set-theoretic description of the Hilbert--Chow morphism
 $\Hilb^{[n]}(\mathbb{A}^2)\rightarrow \Sym^n(\mathbb{A}^2)$ given by $[Z] \mapsto \sum_{p\in \mathbb{A}^2} \big(\text{mult}_p(Z)\big) p$ that appears as the lower horizontal isomorphism in the commutative diagram \eqref{eqn:diagFogarty}.
\end{remark}

\subsection{As a fine moduli space}
\label{sec:fine}
 As in the proof of Theorem~\ref{thm:HilbnGIT}, every $\theta$-semistable $\Pi$-module is $\theta$-stable whenever $\theta\neq 0$. Since $(1,n)$ is indivisible, it follows from \cite[Proposition~5.3]{King94} that $\mathfrak{M}_{\theta}(\mathbf{v},\mathbf{w})$ is the fine moduli space of $\theta$-stable $\Pi$-modules of dimension vector $(1,n)$. 
 
 \begin{definition}[\textbf{Tautological bundle}]
 The \emph{tautological bundle} on $\mathfrak{M}_{\theta}(\mathbf{v},\mathbf{w})$ is the vector bundle $\mathcal{O}_{\mathfrak{M}}\oplus \mathcal{V}_0$ such that the fibre over any closed point is the corresponding  $\theta$-stable $\Pi$-module; note that $\rank(\mathcal{V}_0) = n$. The \emph{tautological $\kk$-algebra homomorphism} $\Pi\to \End(\mathcal{O}_{\mathfrak{M}}\oplus \mathcal{V}_0)$ encodes the maps between summands of $\mathcal{O}_{\mathfrak{M}}\oplus \mathcal{V}_0$ determined by arrows in the quiver $Q$ from Figure~\ref{fig:nakajima}.
 \end{definition}
 
 Any flat family of $\theta$-stable $\Pi$-modules of dimension vector $(1,n)$ over a scheme $X$ is the pullback of $\mathcal{O}_{\mathfrak{M}}\oplus \mathcal{V}_0$ via a unique morphism $X\to \mathfrak{M}_{\theta}(\mathbf{v},\mathbf{w})$. The only difference from the moduli space description  from Remark~\ref{rem:fine1} is the inclusion here of the zero map $j\colon \mathcal{V}_0\to \mathcal{O}_{\mathfrak{M}}$. 
 

This fine moduli interpretation of $\mathfrak{M}_{\theta}(\mathbf{v},\mathbf{w})$ allows for a more algebraic proof of Theorem~\ref{thm:Fogarty}. 

 \begin{proof}[Alternative proof of Theorem~\ref{thm:Fogarty}, paragraph one]
 The image of the class of $b^*$ under the tautological $\kk$-algebra homomorphism $\Pi\to \End(\mathcal{O}_\mathfrak{M}\oplus \mathcal{V}_0)$ is a map of vector bundles $\mathcal{V}_0\to \mathcal{O}_\mathfrak{M}$. Since $j=0$ for every every $\chi_1$-stable point $(B_1,B_2,i,j)\in \mu^{-1}(0)$, this map is zero on each fibre, and since $\mathfrak{M}_1(\mathbf{v},\mathbf{w})$ is smooth and hence reduced, the map $\mathcal{V}_0\to \mathcal{O}_\mathfrak{M}$ is zero. It follows that $\Pi\to \End(\mathcal{O}_\mathfrak{M}\oplus \mathcal{V}_0)$ factors through $\Pi/(b^*)\cong \Lambda$, so $\mathcal{O}_\mathfrak{M}\oplus \mathcal{V}_0$ is in fact a flat family of $\theta$-stable $\Lambda$-modules of dimension vector $(1,n)$. The universal property of $\Hilb^{[n]}(\mathbb{A}^2)$ determines a morphism 
 $\iota\colon \mathfrak{M}_1(\mathbf{v},\mathbf{w})\rightarrow \Hilb^{[n]}(\mathbb{A}^2)$ satisfying $\iota^*(\mathcal{T})\cong \mathcal{V}_0$ and $\iota^*(\mathcal{O}_{\Hilb})\cong \mathcal{O}_{\mathfrak{M}}$. For the inverse, lift $\Lambda\to \End(\mathcal{O}_{\Hilb}\oplus \mathcal{T})$ to 
 $\Pi\to \End(\mathcal{O}_{\Hilb}\oplus \mathcal{T})$ by defining the map $\mathcal{T}\to \mathcal{O}_{\Hilb}$ associated to $b^*$ to be the zero map; the relation is satisfied as the maps for $b^*$ and $aa^*-a^*a$ are zero. Then $\mathcal{O}_{\Hilb}\oplus\mathcal{T}$  is a flat family of $\theta$-stable $\Pi$-modules of dimension vector $(1,n)$, and the universal property of
 $\mathfrak{M}_1(\mathbf{v},\mathbf{w})$ induces $\iota^{-1}$.
 \end{proof}
 
 \begin{remark}
 \label{rem:uniqueness}
 The map sending $\eta\in \ZZ$ to the line bundle $\det(\mathcal{V}_0)^{\otimes \eta}$ on $\Hilb^{[n]}(\mathbb{A}^2)$ induces a $\mathbb{Q}$-linear isomorphism\footnote{After identifying $\Pic(\Hilb^{[n]}(\mathbb{A}^2))\otimes_\ZZ \QQ\cong N^1(X/Y)$ for $X= \Hilb^{[n]}(\mathbb{A}^2)$ and $Y=\Sym^n(\mathbb{A}^2)$, this isomorphism is the `linearisation map' that we introduce in Definition~\ref{def:linearisation} below.} $\QQ\cong \Pic(\Hilb^{[n]}(\mathbb{A}^2))\otimes_\ZZ \QQ$ that identifies the ample cone of $\Hilb^{[n]}(\mathbb{A}^2)$ with the positive half-line in $\QQ$. The morphism $\Hilb^{[n]}(\mathbb{A}^2)\to \Sym^n(\mathbb{A}^2)$ obtained by VGIT from $\eta>0$ to $\eta=0$ contracts a divisor, so the closure of the ample cone is the cone of movable divisors on $\Hilb^{[n]}(\mathbb{A}^2)$. If there were a second projective symplectic resolution $X^\prime\to \Sym^n(\mathbb{A}^2)$, then the ample cone of $X^\prime$ would be an open half-line in the movable cone that lies disjoint from the ample cone of $\Hilb^{[n]}(\mathbb{A}^2)$, a contradiction.
 See Example~\ref{exa:A1n2} for a singularity admitting more than one projective symplectic resolution. 
  \end{remark} 

\section{Nakajima quiver varieties for the McKay graph}

We now introduce the necessary background on ADE singularities and Nakajima quiver varieties that we need for the construction of the Hilbert scheme of $n$ points on an ADE singularity. To avoid a proliferation of notation, we recycle the letters $\mathcal{Q}$ and $Q$ for the key quivers in what follows.

\subsection{On ADE singularities}
The set of finite nontrivial subgroups $\Gamma\subset \SL(2,\kk)$ are classified up to conjugation by Dynkin diagrams of type ADE. In each case, the subalgebra $\kk[x,y]^\Gamma$ of $\Gamma$-invariant polynomials is well known to admit three $\kk$-algebra generators and one relation,  giving rise to an isomorphism $\kk[x,y]^\Gamma \cong \kk[u,v,w]/(f)$. The resulting quotient singularity $\mathbb{A}^2/\Gamma = \Spec \: \kk[x,y]^\Gamma$ is therefore isomorphic to the hypersurface $(f=0)\subset \mathbb{A}^3=\Spec \kk[u,v,w]$. To emphasise the classification, we refer to $\mathbb{A}^2/\Gamma$ as the ADE \emph{singularity} associated to $\Gamma$, though in this context it is typically called a \emph{Kleinian} singularity. We summarise the classification in Table~\ref{tab:ADE}.
\begin{table}[!ht]
     \centering
     \begin{tabular}{ccc}
     \underline{Dynkin diagram} & \underline{Finite subgroup $\Gamma$ in $\SL(2,\kk)$} & \underline{Defining polynomial $f\in \kk[u,v,w]$} \\ 
  $A_r\text{ for }r\geq 1$          & cyclic group $\mu_{r+1}$ & $uw-v^{r+1}$ \\
  $D_r\text{ for }r\geq 4$          & binary dihedral group $\mathbb{D}_{4(r-2)}$ & $u^2+v^2w + w^{r-1}$ \\ 
  $E_6$          & binary tetrehedral group $\mathbb{T}_{24}$ & $u^2+v^3+w^4$ \\
  $E_7$          & binary octahedral group $\mathbb{O}_{48}$ & $u^2+v^3+vw^3$ \\
  $E_8$          & binary icosahedral group $\mathbb{I}_{120}$ & $u^2+v^3+w^5$ \\
   & & 
       \end{tabular}
     \caption{Classification and defining equations for ADE singularities}
     \label{tab:ADE}
 \end{table}
 
 \subsection{The preprojective algebra for the McKay graph}
 The \emph{McKay graph} of $\Gamma$ is defined to be the graph whose vertex set is $\text{Irr}(\Gamma)=\{\rho_0, \dots, \rho_r\}$, the set  of isomorphism classes of irreducible representations. The number of edges in the graph joining $\rho_i$ to $\rho_j$ equals $\dim \Hom_\Gamma(\rho_j,\rho_i\otimes V)$, where $V$ is the 2-dimensional representation given by the inclusion $\Gamma\subset \SL(2,\kk)$. A famous result of McKay~\cite{McKay80} establishes that this graph is an affine Dynkin diagram of the same ADE type as the conjugacy class of $\Gamma$, where the extended vertex corresponds to the trivial representation $\rho_0$. In fact, the affine root system $\Phi_{\textrm{aff}}$ of type ADE is determined by the free abelian group
 \[
 \Rep(\Gamma):=\bigoplus_{\rho\in \text{Irr}(G)} \mathbb{Z} \rho
 \]
 with symmetric bilinear form defined by the Cartan matrix $C_\Gamma=2\text{Id}-A_\Gamma$, where $A_\Gamma$ is the adjacency matrix of the McKay graph of $\Gamma$. The irreducible representations provide a system of simple roots, and the regular representation of $\Gamma$ is the minimal imaginary root $\delta$. In particular, we have
  \begin{equation}
      \label{eqn:dimreps}
  2\dim(\rho_k) = \sum_{\ell\sim k} \dim(\rho_\ell)
    \end{equation}
  for each $0\leq k\leq r$, where $\ell\sim k$ indicates that vertex $\rho_\ell$ is adjacent to $\rho_k$ in the McKay graph.   We use the excellent Dynkins diagrams package of McKay\footnote{No, not the same McKay.}~\cite{McKay21} 
  to draw the affine ADE diagrams in Figure~\ref{fig:Dynkin}, indicating next to each vertex the dimension of the corresponding irreducible representation of $\Gamma$; observe that equation \eqref{eqn:dimreps} holds.
 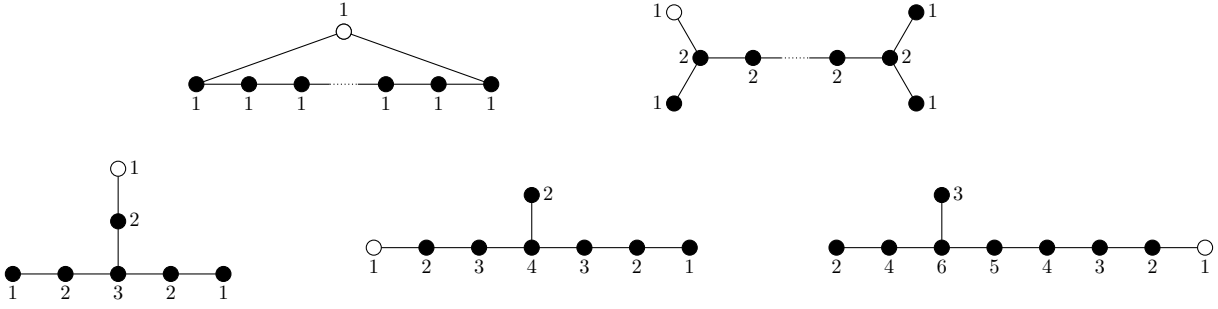
\begin{figure}[ht!]
\centering
\begin{tikzpicture}[yscale=0.77]
\node (M0) at (-3,3) {\dynkin[%
  labels={1,1,1,1,1,1,1},
  scale=2,
  extended] A{***...***}};
\node (M1) at (3,3) {\dynkin[%
  labels={1,1,2,2,2,2,1,1},
  label directions={,,left,,,right,,},
  scale=2,
  extended] D{***...****}};
  \node (M2) at (-6,0) {\dynkin[%
  labels={1,1,2,2,3,2,1},
  scale=2,
  extended] E{******}};
  \node (M2) at (-0.5,0) {\dynkin[%
  labels={1,2,2,3,4,3,2,1},
  scale=2,
  extended] E{*******}};
  \node (M2) at (6,0) {\dynkin[%
  labels={1,2,3,4,6,5, 4, 3, 2},
  scale=2,
  extended] E{********}};
\end{tikzpicture}
\caption{Affine Dynkin diagrams labelled with dimensions of irreducible representations}
\label{fig:Dynkin}
\end{figure}

   The \emph{McKay quiver} $Q^\Gamma$ is the doubled quiver of the McKay graph of $\Gamma$ from Figure~\ref{fig:Dynkin}. That is $Q^\Gamma$ is the quiver with vertex set $\Irr(\Gamma)$, and where the arrow set $Q^\Gamma_1$ is obtained by replacing each edge in the graph by a pair of arrows pointing in opposite directions along the edge. For any $a\in Q^\Gamma_1$, we write $a^*$ for the arrow pointing in the opposite direction, i.e.\ the vertices at the tails and heads satisfy $\tail(a)= \head(a^*)$ and $\head(a)= \tail(a^*)$. 
   Let $\kk Q^\Gamma$ denote the path algebra of $Q^\Gamma$. For $0\leq k\leq r$, write $e_k\in \kk Q^\Gamma$ for the idempotent corresponding to the trivial path at vertex $\rho_k$. Let $\epsilon\colon Q^\Gamma_1\to \{\pm 1\}$ be any map such that $\epsilon(a) \neq \epsilon(a^*)$ for all $a\in Q^\Gamma_1$. The \emph{preprojective algebra} of $\Gamma$ is the noncommutative algebra $\Pi_\Gamma$ obtained as the quotient of $\kk Q^\Gamma$ by the two-sided ideal 
    \begin{equation}
        \label{eqn:PiGammarelations}
    \left\langle \sum_{a\in Q^\Gamma_1 \: : \: \head(a) = \rho_k} \epsilon(a)aa^* \mid 0\leq k\leq r\right\rangle.
     \end{equation}
 
 For $0\leq k\leq r$, let $\rho_k^*$ denote the representation dual to $\rho_k$, and define
 \[
 R_k:= \Hom_{\Gamma}\big(\rho_k,\kk[x,y]\big)\cong \Hom_\Gamma\big(\rho_0,\kk[x,y]\otimes_\kk \rho_k^*\big)\cong \big(\kk[x,y]\otimes_\kk \rho_k^*\big)^\Gamma.
 \]
  Note that $R_0=\kk[x,y]^\Gamma$. The image of the (injective) evaluation map $R_k\otimes \rho_k\to \kk[x,y]$ comprises the direct sum of all $\Gamma$-submodules of $\kk[x,y]$ that are isomorphic to $\rho_k$, and there is an isomorphism  
 \begin{equation}
 \label{eqn:decompkxy}
 \kk[x,y]\cong \bigoplus_{0\leq k\leq r} R_k\otimes_\kk \rho_k
 \end{equation}
 of $R_0$-modules. Each $R_k$ is reflexive as an $R_0$-module, so it is the space of global sections of a coherent sheaf on $\mathbb{A}^2/\Gamma$. The following result therefore provides a geometric interpretation of $\Pi_\Gamma$. 
 
  \begin{lemma}
 \label{lem:preprojEnd}
 There is a $\kk$-algebra isomorphism
 $ \Pi_\Gamma\cong \End_{R_0}\!\big(\bigoplus_{0\leq k\leq r} R_k\big)$.
 \end{lemma}
 \begin{proof}
Let $S\coloneqq\kk[x,y]* \Gamma$ denote the skew group algebra of $\Gamma$; that is, $S$ is the free $\kk[x,y]$-module with basis the elements of $\Gamma$, where multiplication is determined by the rule $tg\cdot t'g' = t(t')^g gg'$ for $t, t'\in \kk[x,y]$ and $g, g'\in \Gamma$.  Work of Auslander~\cite{Auslander62} establishes a $\kk$-algebra isomorphism 
\[
\Phi\colon S\longrightarrow \End_{R_0}\big(\kk[x,y]\big) 
\]
determined by sending $tg$ to the endomorphism satisfying $t'\mapsto t (t')^g$ for $t\in \kk[x,y]$. For $0\leq k\leq r$, choose an idempotent $f_i\in \kk[\Gamma]$ such that $\kk[\Gamma]f_k\cong \rho_k$. Multiplying an element $tg$ in $S$ on the right (resp.\ left) by $f_k$ corresponds under the isomorphism $\Phi$ to precomposing (resp.\ postcomposing) the appropriate endomorphism $\phi$ with projection from $\kk[x,y]$ to the summand $R_k$ from the decomposition \eqref{eqn:decompkxy} that is singled out by our choice of $f_k$. In particular, for any $0\leq k,\ell\leq r$, the isomorphism $\Phi$ identifies $f_\ell S f_k$ with $\Hom_{R_0}(R_k,R_\ell)$. Following Crawley-Boevey--Holland~\cite[Theorem~3.4]{CBH98}, if we set $f=f_0+\cdots + f_r$, then there is a $\kk$-algebra isomorphism $fSf\cong \Pi_\Gamma$ that sends $f_k$ to $e_k$ for each $0\leq k\leq r$. Therefore, for any $0\leq k,\ell\leq r$, we have
\[
\Hom_{R_0}(R_k,R_\ell)\cong f_\ell S_\Gamma f_k = f_\ell(fSf)f_k \cong e_\ell(\Pi_\Gamma)e_k.
\]
The result follows by taking the direct sum of these vector spaces, one for each pair $0\leq k,\ell\leq r$.
 \end{proof}

  
  \begin{remark}
  The construction of the Morita equivalence between $\Pi_\Gamma$ and $S$ goes back to the work of Reiten--Van den Bergh~\cite[Proposition~2.13]{RVdB89}, but the explicit statement of the isomorphism from Lemma~\ref{lem:preprojEnd} is rather hard to find in the literature (see Buchweitz~\cite{Buchweitz12}). 
  \end{remark}

  \subsection{Nakajima quiver varieties for the McKay graph}
 The combinatorial input for a Nakajima quiver variety is a graph and a pair of dimension vectors. Here, we work with the McKay graph of $\Gamma$ and we fix dimension vectors $\mathbf{v}, \mathbf{w}\in \Rep(\Gamma)$ given by $\mathbf{v}:=n\delta, \mathbf{w}:=\rho_0$ for some $n\geq 1$, where $\delta$ and $\rho_0$ are the regular and trivial representations of $\Gamma$ respectively. 
 
 Armed with this input data, our choice of  $\mathbf{w}=\rho_0$ tells us to extend the McKay graph by adding a framing vertex $\infty$ and an edge joining $\infty$ to the vertex $\rho_0$. The doubled quiver $Q$ of this graph, known as the \emph{framed McKay quiver}, has vertex set $\{\infty\}\cup \Irr(\Gamma)$ and replaces each edge of the framed McKay graph by a pair of arrows, one in each direction. Equivalently, starting from the McKay quiver $Q^\Gamma$, we add a pair of arrows $b, b^*$ where $\tail(b)=\head(b^*)=\infty$ and $\tail(b^*)=\head(b)=\rho_0$ . In addition, our choice of $\mathbf{v}:=n\delta$ indicates that we consider representations of $Q$ of dimension vector $v\coloneqq (1,n\delta)\in \NN\oplus \Rep(\Gamma)$; the components of this vector at the vertices other than the framing vertex are obtained by multiplying the numbers from Figure~\ref{fig:Dynkin} by $n$.
 
\begin{example}
The binary dihedral subgroup $\Gamma$ in $\SL(2,\CC)$ of order 12 has McKay graph equal to the affine Dynkin graph of type $D_5$. The framed McKay quiver is shown in Figure~\ref{fig:framedMcKay} and the dimension vector $v=(1,v_{\rho_0},\dots,v_{\rho_5})$ satisfies $v_{\rho_k}=n$ for $k\in \{0,1,4,5\}$ and $v_{\rho_k}=2n$ for $k\in \{2,3\}$.
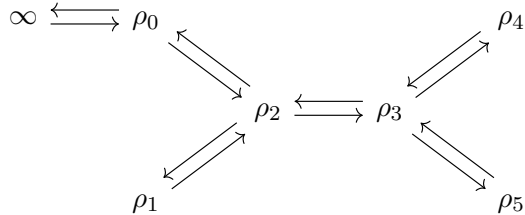
\begin{figure}[!ht]
\[
\begin{tikzcd}
\infty \ar[r,shift right] & \rho_0 \ar[l,shift right] \ar[dr,shift right]& & & \rho_4 \ar[dl,shift right]\\
 &  & \rho_2 \ar[r,shift right] \ar[lu,shift right] \ar[ld,shift right]& \rho_3 \ar[l,shift right] \ar[ur,shift right]\ar[dr,shift right]& \\
 & \rho_1 \ar[ur,shift right] & & & \rho_5\ar[ul,shift right]
 \end{tikzcd}
\]
\caption{The framed McKay quiver for the binary dihedral group of order 12}
\label{fig:framedMcKay}
\end{figure}
\end{example}

 The framed \emph{preprojective algebra} of $\Gamma$ is the noncommutative algebra $\Pi$ obtained as the quotient of the path algebra $\kk Q$ by the ideal of relations
   \begin{equation}
        \label{eqn:Pirelations}
    \left\langle \sum_{a\in Q_1 \: : \: \head(a) = \rho_k} \epsilon(a)aa^* \mid 0\leq k\leq r\right\rangle.
     \end{equation}
     In particular, $\epsilon(b)bb^*$ appears as one of the terms in the generator indexed by $k=0$; this is the only term where a generator from \eqref{eqn:Pirelations} differs from the corresponding generator of \eqref{eqn:PiGammarelations}.

 To construct the quiver varieties in this case, consider the vector space of representations 
 \[
 \Rep(Q,v)= \bigoplus_{a\in Q_1} \Hom_\kk\big(\kk^{v_{\tail(a)}},\kk^{v_{\head(a)}}\big)
 \]
 of $Q$ of dimension vector $v=(1,n\delta)$. Each representation is a tuple
 $(B,i,j)\in \Rep(Q,v)$, where $i\in \kk^n$ and $j\in (\kk^n)^*$ correspond to the arrows $b$ and $b^*$ respectively, and where $B=(B_a)$ is a representation of the McKay quiver $Q^\Gamma$ of dimension vector $n\delta$. The group 
 \[
 G:=\prod_{0\leq k\leq r} \GL(n\dim(\rho_k),\kk)
 \]
 acts on this space of representations, where $g=(g_k)\in G$ acts by
 \[
 g\cdot (B,i,j) = \Big(\big(g_{\head(a)}B_a g_{\tail(a)}^{-1}\big), g_0i, jg_0^{-1}\Big).
\]
 Just as in \eqref{eqn:mu}, consider the map
 $\mu\colon \Rep(Q,v)\to \bigoplus_{0\leq k\leq r} \End(\kk^{n\dim(\rho_k)})$ satisfying
 \begin{equation}
     \label{eqn:moment}
 \mu\big((B,i,j)\big) = \left(ij + \sum_{\head(a)=\rho_0} \epsilon(a)B_aB_{a^*}, \sum_{\head(a)=\rho_1} \epsilon(a)B_aB_{a^*},\dots, \sum_{\head(a)=\rho_r} \epsilon(a)B_aB_{a^*}\right),
 \end{equation}
 where each sum runs over arrows of the McKay quiver $Q^\Gamma$ with head as indicated.

\begin{remarks}
\label{rems:preprojrelations}
\begin{enumerate}
 \item Our ideal \eqref{eqn:Pirelations} defining the framed preprojective relations 
 is generated by cycles corresponding to the components of the map \eqref{eqn:moment}, with one generator for each $0\leq k\leq r$. By defining the map $\mu$ in this way, we have that $(B,i,j)\in\Rep(Q,v)$ lies in $\mu^{-1}(0)$ if and only if the corresponding $\kk Q$-module $V$ of dimension vector $v$ satisfies the preprojective relations \eqref{eqn:Pirelations}, that is, if and only if it is a $\Pi$-module. 
 \item Since 
 $v=(1,n\delta)$ has
 $v_\infty=1$, any point $(B,i,j)\in \mu^{-1}(0)$ satisfies
 \[
 \textrm{Tr}(ji) = \textrm{Tr}(ij) = -\sum_{\head(a)=\rho_0}\epsilon(a) \textrm{Tr}(B_aB_{a^*})= - \sum_{0\leq \head(a)\leq r}\epsilon(a) \textrm{Tr}(B_aB_{a^*}) = 0
 \]
 because $\textrm{Tr}([B_a,B_{a^*}])=0$. Therefore, every such point satisfies $ji = 0$, or equivalently, the corresponding $\Pi$-module $V$ satisfies the additional relation $b^*b=0$. Here, we have chosen not to include the cycle $b^*b$ as a generator of the ideal \eqref{eqn:Pirelations} because we get $ji=0$ for free anyway, but that cycle is in the ideal defining the preprojective relations from \cite[(3.1)]{BC20}. 
\end{enumerate} 
 \end{remarks}

 To define stability conditions, fix the isomorphism $\ZZ^{r+1}\cong G^\vee$ sending $\theta\coloneqq (\theta_0,\dots,\theta_r)\in \ZZ^{r+1}$ to
 \begin{equation}
 \label{eqn:charactergeneral}
 \chi_\theta\colon G\longrightarrow \mathbb{C}^\times, \quad g\mapsto \prod_{0\leq k\leq r} \det(g_k)^{\theta_k}.
 \end{equation}
 Again, some authors use the exponent $-\theta_k$; our choice is explained by Lemma~\ref{lem:ellCdet}.
 Identify $\theta\in \ZZ^{r+1}$ with the map $\theta\colon \ZZ\oplus \Rep(\Gamma)\to\ZZ$ that sends $(v_\infty,\sum_{0\leq k\leq r} v_k\rho_k)$ to the integer
 \[
 - \sum_{0\leq k\leq r} n\dim(\rho_k)\theta_k v_\infty + \sum_{0\leq k\leq r} \theta_k v_k.
 \]
 Notice that for $v=(1,n\delta)$, we have $\theta(v)=0$. For any representation $V$ of dimension vector $v$ and any subrepresentation $V'\subseteq V$, define 
 \[
 \theta(V'):=\theta\big(\dim_\kk V'_\infty, \dim_\kk V'_0, \dots, \dim_\kk V'_r\big).
 \]
 Generalising  Definition~\ref{def:stability},  we say that a representation $V$ is \emph{$\theta$-semistable} (resp.\ \emph{$\theta$-stable}) if every subrepresentation $0\subsetneq V'\subsetneq V$ satisfies $\theta(V')\geq 0$ (resp.\ $\theta(V')>0$). More generally, the notion of $\theta$-stability makes sense for any parameter in the space of of rational stability conditions
  \[
    \Theta\coloneqq \big\{\theta\in \Hom(\ZZ\oplus\Rep(\Gamma),\QQ) \mid \theta(v)=0\big\}\cong \QQ^{r+1}.
    \]
   
  \begin{definition}[\textbf{Nakajima quiver variety for} $Q$] 
 For $\theta\in \Theta$ and the combinatorial data described above, the associated \emph{Nakajima quiver variety} is defined to be
 the GIT quotient
 \[
 \mathfrak{M}_{\theta}(\mathbf{v},\mathbf{w})\coloneqq \mu^{-1}(0)\git_{\chi_\theta} G:= 
\Proj \bigoplus_{k\geq 0} \kk\big[\mu^{-1}(0)\big]_{k\chi_\theta}.
 \]
 Following Bellamy--Schedler~\cite{BS21}, we give $\mathfrak{M}_{\theta}(\mathbf{v},\mathbf{w})$ the reduced scheme structure.
 \end{definition}

\begin{remarks}
\label{rems:Nakrems}
\begin{enumerate}
     \item Building on Crawley-Boevey~\cite{CB03}, the work of Bellamy--Schedler~\cite{BS21} establishes that $\mathfrak{M}_\theta(\mathbf{v},\mathbf{w})$ is a normal and irreducible variety that has  symplectic singularities. In particular, $\mathfrak{M}_\theta(\mathbf{v},\mathbf{w})$ has rational Gorenstein singularities \cite[Proposition~1.3]{Beauville00}.
     \item For $\theta=0$, the corresponding affine quiver variety satisfies 
     \[
     \mathfrak{M}_0(\mathbf{v},\mathbf{w}) \cong \Sym^n(\mathbb{A}^2/\Gamma).
     \]
     Since $\mathbf{v}=n\delta$, this follows by combining the decomposition theorem~\cite[Theorem~1.1]{CB02} and the isomorphism $\mathfrak{M}_0(\delta,\mathbf{w}) \cong \mathbb{A}^2/\Gamma$ that goes back to Kronheimer~\cite{Kronheimer89} over 
     $\mathbb{C}$.
    \item For any $\theta\in \Theta$, variation of GIT induces a projective morphism $   f_\theta\colon \mathfrak{M}_\theta(\mathbf{v},\mathbf{w})\rightarrow  \mathfrak{M}_0(\mathbf{v},\mathbf{w})$. In particular, every such quiver variety is projective over the affine variety $\Sym^n(\mathbb{A}^2/\Gamma)$.
\end{enumerate}
\end{remarks}

\begin{example}
\label{exa:nGamma}
 As we remarked in the introduction, the fact that the space 
\[
 n\Gamma\text{-}\Hilb(\mathbb{A}^2) = \big\{ \Gamma\text{-invariant ideals }I\subset \kk[x,y] \mid \kk[x,y]/I\cong_\Gamma \kk[\Gamma]^{\oplus n}\big\}
 \]
  is isomorphic as a variety to $\mathfrak{M}_\theta(\mathbf{v},\mathbf{w})$ for any stability parameter in the open cone
  \[
 C_+\coloneqq \big\{\theta\in \Theta \mid \theta(\rho_k)>0 \text{ for }0\leq k\leq r\big\}
 \]
 was observed by Varagnolo--Vasserot~\cite{VV99} and Wang~\cite{Wang02}. Indeed, if we write $S=\kk[x,y]*\Gamma$ for the skew group algebra, then each  quotient $\kk[x,y]/I$ is a cyclic $S$-module of dimension vector $n\delta$ which corresponds, via the Morita equivalence $fSf\cong \Pi_\Gamma$ described in the proof of Lemma~\ref{lem:preprojEnd}, to a cyclic $\Pi_\Gamma$-module of dimension vector $n\delta$. Thus,  $n\Gamma\text{-}\Hilb(\mathbb{A}^2)$ parametrises $\Pi_\Gamma$-modules $V=\bigoplus_{0\leq k\leq r} V_k$ of dimension vector $n\delta$, each with a $\kk$-linear map $i\colon \kk\to V_0$ such that $V$ is generated by $i(1)$. Just as in the proof of Theorem~\ref{thm:Fogarty}, every such cyclic $\Pi_\Gamma$ module may be regarded as a $\theta$-stable $\Pi$-module $(B,i,j)$ of dimension vector $(1,n\delta)$ satisfying $j=0$, so $n\Gamma\text{-}\Hilb(\mathbb{A}^2)$ is in fact isomorphic as a variety to $\mathfrak{M}_\theta(\mathbf{v},\mathbf{w})$. Note that $\theta$-stable $\Pi$-modules are often called `$\infty$-generated'.
 \end{example}

\begin{remark}
\label{rem:algebraA}
 We highlight one important point from Example~\ref{exa:nGamma}. The fact that $j=0$ for each $\theta$-stable $\Pi$-module of dimension $(1,n\delta)$ means that we may regard $\mathfrak{M}_\theta(\mathbf{v},\mathbf{w})$ as the fine moduli space of $\theta$-stable modules over the quotient algebra $A:=\Pi/(b^*)$ that have dimension vector $(1,n\delta)$, see \cite[Proposition~3.1]{CGGS19}. Similarly, the preprojective algebra of the (unframed) McKay quiver $\Pi_\Gamma$ is a quotient of $A$, so there are $\kk$-algebra homomorphisms
  \[
\begin{tikzcd}
\Pi \ar[rr,"/(b^*)"] & & A \ar[rr,"/(b)"] & & \Pi_\Gamma.
\end{tikzcd}
\]
In fact, $\Pi_\Gamma$ is also a subalgebra of $A$, namely $\Pi_\Gamma\cong \bigoplus_{0\leq k,\ell\leq r} e_\ell A e_k$.
\end{remark}

 \subsection{Tautological bundles}
 As is standard in GIT, the rational vector space $\Theta$ of GIT stability parameters admits a wall-and-chamber decomposition, where $\theta, \theta'\in \Theta$ lie in the relative interior of the same cone if and only if the notions of $\theta$-semistability and $\theta'$-semistability coincide. The interiors of the top-dimensional cones in $\Theta$ are \emph{chambers}, while the codimension-one faces of the closure of each chamber are called \emph{walls}. We say that $\theta\in \Theta$ is \emph{generic} if it lies in some chamber. Note that our definition of generic differs from that of Nakajima~\cite{Nakajima94}, see Remark~\ref{rem:generic}(1).

Since $v=(1,n\delta)$ is indivisible,  King~\cite[Proposition~5.3]{King94} shows that $\mathfrak{M}_\theta(\mathbf{v},\mathbf{w})$ is the fine moduli space of $\theta$-stable $\Pi$-modules of dimension vector $v$ whenever $\theta\in \Theta$ is generic.

 \begin{definition}[\textbf{Tautological bundle}]
 The \emph{tautological bundle} on $\mathfrak{M}_{\theta}(\mathbf{v},\mathbf{w})$ is the vector bundle 
 \[
 \mathcal{R}\coloneqq \mathcal{O}_{\mathfrak{M}}\oplus \bigoplus_{0\leq k\leq r} \mathcal{R}_k
 \]
 whose fibre over any closed point of $\mathfrak{M}_\theta(\mathbf{v},\mathbf{w})$ is the corresponding $\theta$-stable  $\Pi$-module; note that $\rank(\mathcal{R}_k) = n\dim(\rho_k)$ for $0\leq k\leq r$. The \emph{tautological $\kk$-algebra homomorphism} $\Pi\to \End(\mathcal{R})$ encodes the maps between summands of $\mathcal{R}$ determined by arrows in $Q$.
 \end{definition}
 

 \begin{remark}
 As before, $\mathfrak{M}_{\theta}(\mathbf{v},\mathbf{w})$ is universal, in the sense that any flat family of $\theta$-stable $\Pi$-modules of dimension $(1,n)$ over a scheme $X$ is obtained as the pullback of the tautological bundle $\mathcal{R}$ via a unique morphism $X\to \mathfrak{M}_{\theta}(\mathbf{v},\mathbf{w})$.
 \end{remark}
 
 Let $C\subseteq \Theta$ be a GIT chamber. For $\theta\in C$, write $\pi\colon \mu^{-1}(0)^{\chi_\theta\textrm{-s}}\to \mathfrak{M}_\theta(\mathbf{v},\mathbf{w})$ for the geometric quotient of the locus of $\chi_\theta$-stable points of $\mu^{-1}(0)$ by the action of $G$. For any $\chi\in G^\vee$, consider the $G$-equivariant line bundle $\chi\otimes \mathcal{O}$ given by the trivial line bundle on $\mu^{-1}(0)^{\chi_\theta\textrm{-s}}$ equipped with the action of $G$ on each fibre given by $\chi$. Our sign convention on characters $\chi_\theta$ from \eqref{eqn:charactergeneral} is chosen so that we don't need to introduce minus signs in the statement of the following result.

 \begin{lemma}
 \label{lem:ellCdet}
For any $\chi_\eta\in G^\vee$, the $G$-equivariant line bundle $\chi_\eta\otimes \mathcal{O}$ on $\mu^{-1}(0)^{\chi_\theta\textrm{-s}}$ descends via $\pi$ to the line bundle 
\[
L_C(\eta)=\bigotimes _{0\leq k\leq r} \det(\mathcal{R}_k)^{\otimes \eta_k}
\]
on $\mathfrak{M}_\theta(\mathbf{v},\mathbf{w})$. In particular, the polarising ample line bundle on $\mathfrak{M}_\theta(\mathbf{v},\mathbf{w})$ is $L_C(\theta)$.
 \end{lemma}
 \begin{proof}
 For $0\leq k\leq r$, define a representation $\rho_k\colon G\to \GL\big(n\dim(\rho_k)\big)$ and a character $\chi_k\in G^\vee$ by setting $\rho_k(g) = g_k$ and $\chi_k(g) = \det(g_k)$ for $g=(g_k)\in G$. The $G$-equivariant vector bundle $\rho_k\otimes \mathcal{O}$ on $\mu^{-1}(0)^{\chi_\theta\textrm{-s}}$ is the bundle $\mathcal{O}^{\oplus n\dim(\rho_k)}$ equipped with the action of $G$ on each fibre given by $\rho_k$, so $\det(\rho_k\otimes \mathcal{O})$ is isomorphic to $\mathcal{O}$ equipped with the $G$-action given by $\chi_k$. This gives an isomorphism $\det(\rho_k\otimes \mathcal{O})\cong \chi_k\otimes \mathcal{O}$ of $G$-equivariant line bundles. Since $\rho_k\otimes \mathcal{O}$ descends to $\mathcal{R}_k$ on $\mathcal{M}_\theta(\mathbf{v},\mathbf{w})$, we have the following isomorphisms 
 \[
 \chi_k\otimes \mathcal{O}\cong \det(\rho_k\otimes \mathcal{O})\cong \det\big(\pi^*(\mathcal{R}_k)\big)\cong \pi^*\big(\det(\mathcal{R}_k)\big)
 \]
 of $G$-equivariant line bundles. Uniqueness of descent shows that $\chi_k\otimes \mathcal{O}$ descends to $\det(\mathcal{R}_k)$. The first statement follows by $\ZZ$-linearity, while the second follows because $\chi_\theta\otimes \mathcal{O}$ descends to the polarising ample bundle on $\mathfrak{M}_\theta(\mathbf{v},\mathbf{w})$.
 \end{proof}

\section{Birational geometry of Nakajima quiver varieties}
We now review briefly the joint work of the author and Gwyn Bellamy~\cite{BC20} that describes explicitly the GIT wall-and-chamber decomposition of $\Theta$, leading to 
a complete understanding of the birational geometry
of the quiver varieties $\mathfrak{M}_\theta(\mathbf{v},\mathbf{w})$ for all $\theta\in \Theta$. In this section only, it is convenient to identify $\Theta = \Hom(\Rep(\Gamma),\QQ)$ by projecting away from the component indexed by $\infty$.

\subsection{The GIT chamber decomposition}
 For any $\gamma\in \Rep(\Gamma)$ we write $\gamma^\perp=\{\theta \in \Theta \mid \theta(\gamma)=0\}$ for the dual hyperplane. A simple example of the main results is illuminating.

 
\begin{example}\textbf{[Type $A_1$ with $n=2$]}
\label{exa:A1n2}
The group $\Gamma$ gives $r=1$ and $\delta=\rho_0+\rho_1$. Choose $\{\rho_0, \delta\}$ as basis for $\Rep(\Gamma)$.
\begin{figure}[!ht]
			\begin{tikzpicture}[scale=0.5]
				\shade[left color=red,right color=white] (0,0) -- (6.7,0) -- (4.7,4.7)-- cycle;
			\shade[bottom color=green, top color=white] (0,0) -- (4.7,4.7) -- (0,4.7)-- cycle;
			\draw[thick,>=stealth] (-7,0) -- (7,0) node[right] {$\rho_1^\perp$};
			\draw[thick,>=stealth] (-5,-5) -- (5,5) node[right] {$\rho_0^\perp=(\delta-\rho_1)^\perp$};
		
			\draw (5,2) node {$C_+$};
			\draw (2,4) node {$C_-$};
    (0,0) -- (5,0) arc (0:90:5) -- cycle;
			\draw[thick,>=stealth] (0,-5) -- (0,5) node[above] {$\delta^\perp$};
			\draw[thick,>=stealth] (-5,5) -- (5,-5) node[right] {$(\delta+\rho_1)^\perp$};
			\end{tikzpicture}
			\caption{GIT chamber decomposition of $\Theta_v$ for type $A_1$ and $n=2$}
			\label{fig:GITdecomp}
\end{figure}
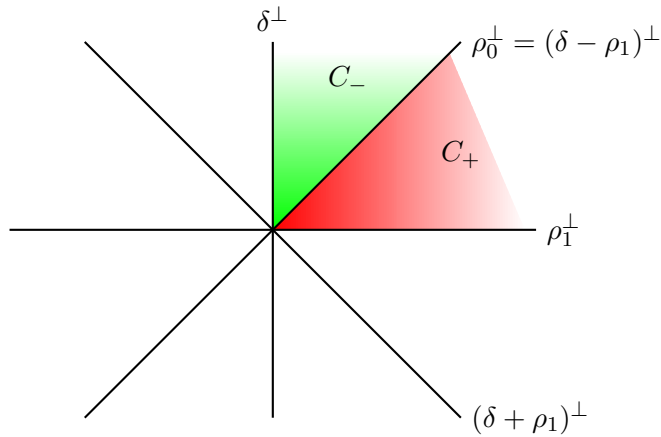
 The GIT chamber decomposition of $\Theta$ is shown in Figure~\ref{fig:GITdecomp}: the red chamber $C_+$ is that from Example~\ref{exa:nGamma} for which $\mathfrak{M}_\theta(\mathbf{v},\mathbf{w})\cong 2\Gamma\text{-}\Hilb(\mathbb{A}^2)$; the green chamber $C_-$ defines an alternative projective crepant resolution of $\Sym^2(\mathbb{A}^2/\Gamma)$, namely $\mathfrak{M}_{\theta'}(\mathbf{v},\mathbf{w})\cong \Hilb^{[2]}(S)$ for $\theta'\in C_-$, where $S\to \mathbb{A}^2/\Gamma$ is the minimal resolution. Notice that every chamber in $\Theta$ is the image of either $C_+$ or $C_-$ under the action of the group $W\cong \ZZ/2 \times \ZZ/2$ generated by the reflections in the supporting hyperplanes $\delta^\perp$ and $\rho_1^\perp$ of the closed cone $F\coloneqq \big\{\theta\in \Theta \mid \theta(\delta)\geq 0, \theta(\rho_1)\geq 0\big\}$.
  \end{example}
 
 To describe the GIT decomposition in general, recall that the ADE root system of finite type $\Phi\subset \Phi_{\textrm{aff}}$ is the intersection of the affine root system $\Phi_{\textrm{aff}}$ of type ADE with the $\ZZ$-span of $\{\rho_1, \dots, \rho_r\}$ in $\Rep(\Gamma)$. Let $\Phi^+$ denote the set of positive roots. The following result is \cite[Theorem~1.1]{BC20}. 
 
 \begin{theorem}[\textbf{GIT chamber decomposition}]
 \label{thm:GITdecomp}
 For $n>1$, the GIT wall-and-chamber decomposition is the partition of $\Theta$ determined by the hyperplane arrangement
 \[
 \mathcal{A}:=\big\{\delta^\perp, (m\delta\pm\alpha)^\perp \mid 0\leq m < n\text{ and }\alpha\in \Phi^+\big\}.
 \]
 Moreover, the projective morphism $f_\theta\colon \mathfrak{M}_\theta(\mathbf{v},\mathbf{w})\to \mathfrak{M}_0(\mathbf{v},\mathbf{w})\cong \Sym^n(\mathbb{A}^2/\Gamma)$ is a symplectic (hence crepant) resolution if and only if $\theta$ does not lie in any of these hyperplanes.
 \end{theorem}
\begin{proof}[Sketch of the proof]
 Following Bellamy--Schedler~\cite[Corollary~1.17]{BS21}, non-singularity of $\mathfrak{M}_\theta(\mathbf{v},\mathbf{w})$ can be characterised in terms of the canonical decomposition of $v=(1,n\delta)$ with respect to $\theta$. An explicit calculation of decompositions of $v$ leads to the vectors defining the hyperplanes in $\mathcal{A}$. 
 \end{proof}
 
 \begin{remarks}
 \label{rem:generic}
 \begin{enumerate}
     \item Nakajima~\cite{Nakajima94} defines $\theta\in \Theta$ to be generic if it lies in the complement of all hyperplanes $\gamma^\perp$ associated to vectors $\gamma\in \Rep(\Gamma)$ satisfying $0 < \gamma \leq \mathbf{v}$ and $\gamma^t C_\Gamma\gamma\leq 2$, where $C_\Gamma$ is the Cartan matrix. Every such $\theta$ is generic in our sense \cite[Theorem~5.2.2(ii)]{Ginzburg12}, but the converse is false. For example, the vector $\gamma :=2\rho_0+\rho_1=2\delta-\rho_1$ in Example~\ref{exa:A1n2} satisfies the above inequalities and yet it does not define a hyperplane from $\mathcal{A}$.
     \item Notice that the supporting hyperplanes of the closed simplicial cone \begin{equation}
         \label{eqn:F}
 F\coloneqq \big\{\theta\in \Theta \mid \theta(\delta)\geq 0 \text{ and }\theta(\rho_k)\geq 0 \text{ for }1\leq k\leq r\big\}
  \end{equation}
 arise in the hyperplane arrangement from Theorem~\ref{thm:GITdecomp}, so $F$ is the union of the closures of a collection of GIT chambers including $C_+$ from Example~\ref{exa:nGamma}, as well as the chamber
 \begin{equation}
 \label{eqn:HilbnS}
 C_- = \big\{\theta\in \Theta \mid \theta(\delta)\geq 0 \text{ and } \theta(m\delta - \alpha)\geq 0 \text{ for }0\leq m<n \text{ and }\alpha\in \Phi^+\big\}
 \end{equation}
 that satisfies $\mathfrak{M}_{\theta'}(\mathbf{v},\mathbf{w})\cong \Hilb^{[n]}(S)$ for $\theta'\in C_-$, where $S\to \mathbb{A}^2/\Gamma$ is the unique minimal resolution (see Kuznetsov~\cite{Kuznetsov07}). In fact, $F$ is generated by the vectors from $C_+\cup C_-$.
 \end{enumerate}
 
 \end{remarks}
 
 \begin{definition}[\textbf{Namikawa Weyl group}]
 The \emph{Namikawa Weyl group} $W$ is the group generated by reflections in the supporting hyperplanes $\delta^\perp, \rho_1^\perp, \dots, \rho_r^\perp$ of the cone $F$.
 \end{definition}

  It's not hard to show that $W\cong \ZZ/2\times W_\Gamma$ where $W_\Gamma$ is the Weyl group of $\Phi$, and moreover, for every GIT chamber $C$ in $\Theta$, there exists a unique $w\in W$ such that $w(C)\subseteq F$  \cite[Proposition~2.2]{BC20}. The importance of this group from a geometric point of view is the existence of isomorphisms
 \begin{equation}
     \label{eqn:Weylgroup}
 \mathfrak{M}_\theta(\mathbf{v},\mathbf{w})\cong\mathfrak{M}_{w(\theta)}(\mathbf{v},\mathbf{w}) \quad\text{ for all }\theta\in \Theta, w\in W;
 \end{equation}
 see \cite[Corollary~1.8]{BC20}. As a result, it suffices to understand the Nakajima quiver varieties defined by stability parameters $\theta\in F$.


 \subsection{Birational geometry of Nakajima quiver varieties}
 The central role played by the cone $F$ from \eqref{eqn:F} becomes clear only when we study its relation to the movable cone. First, let's fix one chamber, say $C_+$ from Example~\ref{exa:nGamma} (we chose $C_-$ in \cite{BC20}), and in doing so we fix once and for all one projective crepant resolution 
 \[
 X:=\mathfrak{M}_\theta(\mathbf{v},\mathbf{w})\cong n\Gamma\text{-}\Hilb(\mathbb{A}^2)\quad\text{for }\theta\in C_+
 \]
 of $Y:=\Sym^n(\mathbb{A}^2/\Gamma)$. Let $N^1(X/Y)$ denote the rational vector space of line bundles up to numerical equivalence (with respect to proper curves in $X$). Recall that $L\in N^1(X/Y)$ is said to be \emph{movable} if the stable base locus of $L$ over $Y$ has codimension at least two in $X$, in which case the rational map
\[
\varphi_{\vert L\vert}\colon X\dashrightarrow X(L)=\Proj \bigoplus_{k\geq 0} H^0(X,L^{\otimes k})\subseteq \vert L\vert
\]
 over $Y$ is an isomorphism in codimension-one. The relative \emph{movable cone} $\Mov(X/Y)$ is the closure in $N^1(X/Y)$ of the cone generated by all movable divisor classes over $Y$. Since the ample cone $\Amp(X/Y)$ is the interior of the nef cone $\Nef(X/Y)$ by Kleiman's relative ampleness criterion, it follows that the movable cone contains the nef cone.

 \begin{definition}[\textbf{The linearisation map}]
 \label{def:linearisation}
 Let $C\subseteq \Theta$ be a GIT chamber. The \emph{linearisation map} is the $\QQ$-linear map
 \[
 L_C\colon \Theta\longrightarrow N^1(X/Y), \quad \eta\mapsto \bigotimes _{0\leq k\leq r} \det(\mathcal{R}_k)^{\otimes \eta_k}.
 \]
\end{definition}

 For $\theta\in C$, the line bundle $L_C(\theta)$ is the polarising ample bundle on $\mathfrak{M}_\theta(\mathbf{v},\mathbf{w})$ by Lemma~\ref{lem:ellCdet}. Since varying the stability parameter $\theta$ from $C$ into any wall of $C$ induces a contraction from the non-singular quiver variety $\mathfrak{M}_\theta(\mathbf{v},\mathbf{w})$ to a singular quiver variety (assuming $n>1$) by Theorem~\ref{thm:GITdecomp}, it follows that $L_C$ restricts to give an isomorphism
 \begin{equation}
    \label{eqn:ample}
    {L_C}\vert_{\overline{C}}\colon \overline{C}\longrightarrow \Nef(\mathfrak{M}_\theta\big(\mathbf{v},\mathbf{w})\big)
\end{equation}
 of finitely generated, rational polyhedral cones that identifies the chamber $C$ with the ample cone of $\mathfrak{M}_\theta(\mathbf{v},\mathbf{w})$ for $\theta\in C$ \cite[Proposition~6.1]{BC20}. 
 
 This GIT interpretation of the ample cone can be strengthened considerably to a description of the entire movable cone.  The main result of \cite[Theorem~1.2]{BC20} can be stated as follows:
 
 \begin{theorem}[\textbf{The movable cone of $\mathfrak{M}_\theta(\mathbf{v},\mathbf{w})$}]
 \label{thm:BC20main} 
 Let $n > 1$ and let $\Gamma\subset \SL(2,\kk)$ be a finite subgroup. There is an isomorphism $L_F\colon \Theta\rightarrow N^1(X/Y)$ of rational vector spaces such that:
 \begin{enumerate}
     \item[\one] $L_F=L_C$ for each chamber $C\subseteq F$; and 
     \item[\two] $L_F$ identifies the GIT wall-and-chamber structure of $F$ with the decomposition of the movable cone of $\mathfrak{M}_\theta(\mathbf{v},\mathbf{w})$ over $\Sym^n(\mathbb{A}^2/\Gamma)$ into Mori chambers.
\end{enumerate}
 Thus, for each $\theta\in F$, the quiver variety $\mathfrak{M}_\theta(\mathbf{v},\mathbf{w})$ is isomorphic to the birational model $X(L_F(\theta))$; and conversely, every partial crepant resolution of $\Sym^n(\mathbb{A}^2/\Gamma)$ arises as $\mathfrak{M}_\theta(\mathbf{v},\mathbf{w})$  for some $\theta\in F$.
 \end{theorem}
 \begin{proof}[Sketch of the proof]
 We analyse how the quiver variety $\mathfrak{M}_\theta(\mathbf{v},\mathbf{w})$ and its tautological bundles $\mathcal{R}_k$ (for $0\leq k\leq r$) change as we vary $\theta$ across GIT walls that touch $F$. Fix $\theta_+\in C\subset F$ and let $\theta'$ lie in the relative interior of a wall of $C$. The key input is an \'{e}tale local description of the contraction morphism $\mathfrak{M}_{\theta^+}(\mathbf{v},\mathbf{w})\to \mathfrak{M}_{\theta^0}(\mathbf{v},\mathbf{w})$ induced by variation of GIT quotient $\theta^+\rightsquigarrow \theta^0$ that is given in terms of the Ext-quiver of a point in $\mathfrak{M}_{\theta^0}(\mathbf{v},\mathbf{w})$, see \cite[Theorem~3.2]{BC20}. The wall lies either:
 \begin{enumerate}
     \item in the interior of $F$. Then variation of GIT quotient across the wall is a Mukai flop of type $A$, 
     given \'{e}tale-locally by a pair of small resolutions of the closure of the nilpotent orbit of matrices of fixed rank that square to zero \cite[Theorem~5.5]{BC20}:
     \begin{equation}
\begin{tikzcd}
\label{eqn:flop}
 \mathfrak{M}_{\theta^+}(\mathbf{v},\mathbf{w}) \ar[rr,"\psi",dashed] \ar[dr,swap,"f_{\theta_+}"] & & \mathfrak{M}_{\theta^-}(\mathbf{v},\mathbf{w}) \ar[dl,"f_{\theta_-}"] \\
& \mathfrak{M}_{\theta^0}(\mathbf{v},\mathbf{w}) & 
\end{tikzcd}
\end{equation}
In addition, the unstable locus for each contraction $f_{\theta^\pm}$ has codimension at least two, so for each $0\leq k\leq r$, the strict transform across $\psi$ of the line bundle $\det(\mathcal{R}^-_k)$ on $\mathfrak{M}_{\theta^-}(\mathbf{v},\mathbf{w})$ coincides with the line bundle $\det(\mathcal{R}^+_k)$ defined in terms of the corresponding tautological bundle on $\mathfrak{M}_{\theta^+}(\mathbf{v},\mathbf{w})$. It follows that the linearisation maps defined by chambers on either side of the wall agree. Thus, the linear maps $L_C$ determined by chambers $C$ in $F$ all agree, and we may define $L_F:=L_C$ for any chamber $C\subset F$.
     \item in the boundary of $F$. The \'{e}tale-local description of the VGIT morphism reveals that it's a divisorial contractions. For example, a unique chamber in $F$ has a wall in $\delta^\perp$, namely the chamber $C_-$ from \eqref{eqn:HilbnS}, and the VGIT morphism recovers the Hilbert--Chow morphism $\Hilb^{[n]}(S)\to \Sym^n(S)$ which is a divisorial contraction.
 \end{enumerate}
 Combining the equalities $L_F=L_C$ for each chamber $C\subset F$ with the isomorphisms \eqref{eqn:ample} establishes that $L_F$ identifies the GIT chamber decomposition of $F$ with the decomposition of the movable cone into Mori chambers. Thus, we probe all partial crepant resolutions of $\Sym^n(\mathbb{A}^2/\Gamma)$ by variation of GIT quotient on $\mathfrak{M}_\theta(\mathbf{v},\mathbf{w})$ for $\theta$ in $F$.
 \end{proof}
 
 \begin{remarks}
 \begin{enumerate}
 \item The statement of Theorem~\ref{thm:BC20main} holds if we replace $F$ by any of its translates $wF$ by some element $w\in W$ in the Namikawa Weyl group; see \cite[Theorem~1.7]{BC20}. It follows that Theorem~\ref{thm:BC20main} determines the geometry of $\mathfrak{M}_\theta(\mathbf{v},\mathbf{w})$ for every $\theta\in \Theta$.
 \item Theorem~\ref{thm:BC20main}  implies more-or-less immediately that $X= \mathfrak{M}_\theta(\mathbf{v},\mathbf{w})$ for $\theta\in C_+$ is a relative Mori Dream Space over $Y=\Sym^n(\mathbb{A}^2/\Gamma)$, see \cite[Corollary~6.5]{BC20}.
 \item When $n=1$, the map $L_C$ is surjective with a one-dimensional kernel, and $\Theta$ is the product of $\QQ$ with the Weyl chamber decomposition of $\delta^\perp$ for the root system $\Phi$ of type ADE. For the full statement, see \cite[Proposition~7.11]{BC20}.
 \end{enumerate}
 \end{remarks}

\section{How to calculate \texorpdfstring{$\Hilb^{[n]}(\mathbb{A}^2/\Gamma)$}{kleinian}}
\label{sec:CGGS19}
 We now have all the elements in place to describe how the results for $\Hilb^{[n]}(\mathbb{A}^2)$ from Section~\ref{sec:HilbnA2} can be generalised to the Hilbert scheme of $n$ points on any ADE surface singularity. 

 \subsection{As a set} 
 For $n\geq 1$, the \emph{Hilbert scheme of $n$ points in $\mathbb{A}^2/\Gamma$} parametrises subschemes of length $n$ in $\mathbb{A}^2/\Gamma$, or equivalently, ideals in $\kk[x,y]^\Gamma$ for which the quotient has dimension $n$ over $\kk$:
\begin{eqnarray}
\Hilb^{[n]}(\mathbb{A}^2/\Gamma) & := & \big\{Z\subset \mathbb{A}^2/\Gamma \mid \dim_{\kk} H^0(\mathcal{O}_Z) = n\big\} \nonumber \\
& = & \big\{ I\subset \kk[x,y]^\Gamma \mid \dim_{\kk} \kk[x,y]^\Gamma/I = n\big\}.\label{eqn:HilbnADE} 
 \end{eqnarray}
 \noindent For example, in type $A_2$ with $n=2$, the group $\Gamma\cong \ZZ/3$  is generated by the diagonal matrix $\text{diag}(\varepsilon,\varepsilon^2)$ for $\varepsilon = e^{2\pi i/3}$.  The ideal $I =\langle x^6, xy, y^3\rangle$
satisfies $\kk[x,y]^\Gamma/I = \kk\langle 1,x^3\rangle$, so $[I]\in \Hilb^{[2]}(\mathbb{A}^2/\Gamma)$.

 \subsection{As a variety}
To establish the analogue of Theorem~\ref{thm:HilbnGIT} for $\Hilb^{[n]}(\mathbb{A}^2/\Gamma)$, we now introduce a quiver $\mathcal{Q}$ similar to that of Figure~\ref{fig:quiverQ'} using the presentation of $\kk[x,y]^{\Gamma}$. Define the vertex set $\mathcal{Q}_0=\{\infty,0\}$ and arrow set $\mathcal{Q}_1=\{b,a_1,a_2,a_3\}$ as shown in Figure~\ref{fig:quiverQ3loops}.  We impose the relations $[a_k,a_\ell]=0$ for $1\leq k,\ell\leq 3$ and $f(a_1,a_2,a_3)=0$, where $f\in \kk[u,v,w]$ is the defining equation for the ADE singularity $\mathbb{A}^2/\Gamma$ listed in Table~\ref{tab:ADE}.
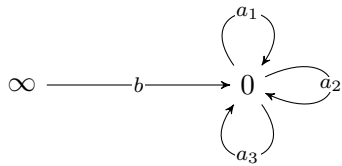
\begin{figure}[h!]
\centering
\begin{tikzpicture}[yscale=0.77]
\node (M0) at (-1,0) {$\infty$};
\node (M1) at (2,0) {$0$};
\draw [->] (M0) to node[gap] {$\scriptstyle{b}$} (M1);
\draw [->,looseness=12, out=120, in=60] (M1) to node[gap] {$\scriptstyle{a_1}$} (M1);
\draw [->,looseness=10, out=30, in=-30] (M1) to node[gap] {$\scriptstyle{a_2}$} (M1);
\draw [->,looseness=12, out=-60, in=-120] (M1) to node[gap] {$\scriptstyle{a_3}$} (M1);
\end{tikzpicture}
\caption{The quiver $\mathcal{Q}$ with relations $[a_k,a_\ell]=0$ $(1\leq k, \ell\leq 3)$ and $f(a_1,a_2,a_3)=0$.}
\label{fig:quiverQ3loops}
\end{figure}
  The point here is that a representation $V$ of $\mathcal{Q}$ with $V_\infty=\kk$ and $V_0=\kk^n$ satisfying the relations encodes a linear map $i\in\Hom(\kk,\kk^n)$ and three commuting endomorphisms $B_1, B_2, B_3\in \End(\kk^n)$ satisfying the additional relation $f(B_1,B_2,B_3) = 0$, so $V_0$ becomes a $\kk[x,y]^\Gamma$-module. If we now define
 \[
 Z:=\Big\{(B_1, B_2, B_3, i)\in \End(\kk^n)^{\oplus 3} \oplus \kk^n \mid [B_k,B_\ell]=0\text{ for }1\leq k, \ell\leq 3\text{ and }f(B_1,B_2,B_3)=0\Big\}
\]
 and let the group $\GL(n,\kk)$ act naturally on $\kk^n$, then the induced action on $Z$ satisfies 
 \[
 g\cdot (B_1,B_2,B_3,i) = \big(gB_1g^{-1},gB_2g^{-1},gB_3g^{-1},gi\big)
 \]
 for $g\in\GL(n,\kk)$. As before, stability conditions are $\theta\colon \ZZ^2\to \ZZ$ satisfying $\theta(v_\infty,v_0) = -n\theta v_\infty + \theta v_0$ for some $\theta\in \ZZ$,  each corresponding to the character $\chi_\theta\in \GL(n,\kk)^\vee$ from \eqref{eqn:character}. The analogue of Theorem~\ref{thm:HilbnGIT} in this situation can be stated as follows.

\begin{theorem}[\textbf{GIT construction of $\Hilb^{[n]}(\mathbb{A}^2/\Gamma)$}]
  \label{thm:HilbnADEGIT}
  Let $\chi\in \GL(n,\kk)^\vee$ be such that 
  $\chi(g) = \det(g)$ for all $g\in \GL(n,\kk)$. Then $\Hilb^{[n]}(\mathbb{A}^2/\Gamma)$ can be constructed as the GIT quotient $Z\git_{\chi} \GL(n,\kk)$.
 \end{theorem}
 
 \begin{proof}
 We've already done the hard work. It remains to show that $i(1)$ generates the $n$-dimensional vector space $V$ as a $\kk[x,y]^\Gamma$-module if and only if the corresponding point $(B_1, B_2,B_3,i)\in Z$ is $\chi$-stable. The proof can be copied verbatim from the proof of Theorem~\ref{thm:HilbnGIT}.
 \end{proof}

 The category of representations of $\mathcal{Q}$ satisfying the relations from Figure~\ref{fig:quiverQ3loops} is equivalent to the category of left-modules over the noncommutative algebra 
 \begin{equation}
     \label{eqn:A0}
A_0:= \kk \mathcal{Q} / \Big\langle [a_1,a_2], [a_1,a_3], [a_2,a_3], f(a_1,a_2,a_3)\Big\rangle.
 \end{equation}
  Write $e_\infty\in A_0$ for the idempotent corresponding to the class of the trivial path at vertex $\infty$. Recall the $\kk$-algebra $R_0\coloneqq\kk[x,y]^\Gamma$.
 
 \begin{lemma}
 \label{lem:A0}
 The $\kk$-algebra $A_0$ is isomorphic to $R_0\oplus R_0b\oplus \kk e_\infty$ with multiplication given by 
 \[
 (\alpha_1,\beta_1b,\gamma_1e_\infty)\cdot (\alpha_2,\beta_2b,\gamma_2e_\infty) = \big(\alpha_1\alpha_2,(\alpha_1\beta_2+\beta_1\gamma_2)b,\gamma_1\gamma_2e_\infty\big)
 \]
 \end{lemma}
 \begin{proof}
 Write $e_0\in A_0$ for the idempotent corresponding to the class of the trivial path at vertex $0$. Then $e_0A_0e_0\cong \kk[u,v,w]/(f)\cong R_0$. The result follows from the fact that the only paths in $\mathcal{Q}$ that touch vertex $\infty$ are the trivial path $e_\infty$ and any path whose first arrow is $b$.
\end{proof}

\begin{remark}
\label{rem:fine3}
As in Remark~\ref{rem:fine1}, for $\theta=1$, we conclude that $\Hilb^{[n]}(\mathbb{A}^2/\Gamma)$ is isomorphic to the fine moduli space $\mathcal{M}(A_0)$ of $\theta$-stable $A_0$-modules of dimension vector $(1,n)$. In particular, it carries a tautological bundle $\mathcal{O}_{\Hilb}\oplus \mathcal{T}$ and a tautological $\kk$-algebra homomorphism $A_0\to \End(\mathcal{O}_{\Hilb}\oplus \mathcal{T})$, and it satisfies a universal property that the reader can by now formulate with confidence. 
\end{remark}

\subsection{As a Nakajima quiver variety}
 We can now state the main result describing $\Hilb^{[n]}(\mathbb{A}^2/\Gamma)$ with the reduced structure as a Nakajima quiver variety. The naive generalisation of  Theorem~\ref{thm:Fogarty}, where one might hope to simply add one extra arrow to the quiver from Figure~\ref{fig:quiverQ3loops} for a new GIT construction, cannot possibly provide the right approach because the resulting quiver would not be the doubled quiver associated to any graph and therefore the GIT quotients would not be Nakajima quiver varieties.
 
 The key question that led to the main result was formulated by Bal\'{a}zs Szendr\H{o}i:
 
\begin{question}
\label{que:szendroi}
 Is $\Hilb^{[n]}(\mathbb{A}^2/\Gamma)$ isomorphic to a Nakajima quiver variety $\mathfrak{M}_\theta(\mathbf{v},\mathbf{w})$ associated to the framed McKay quiver of $\Gamma$ for $\mathbf{v}=n\delta$ and $\mathbf{w}=\rho_0$ with respect to some parameter $\theta_0$? 
\end{question}

 This question was posed following a talk by the author on the results of \cite{BC20}, and indeed the statement of Theorem~\ref{thm:BC20main} makes the question very natural: the Hilbert--Chow morphism from $\Hilb^{[n]}(\mathbb{A}^2/\Gamma)$ makes it a prime candidate to be a partial crepant resolution of $\Sym^n(\mathbb{A}^2/\Gamma)$, and hence a Nakajima quiver variety. As we'll see, the answer to Question~\ref{que:szendroi} is `Yes', at least if we give $\Hilb^{[n]}(\mathbb{A}^2/\Gamma)$ the reduced scheme structure. 
 
 To motivate our approach and our choice of the appropriate $\theta_0$, consider the map  
 \[
\tau\colon n\Gamma\text{-}\Hilb(\mathbb{A}^2)\longrightarrow 
\Hilb^{[n]}(\mathbb{A}^2/\Gamma)
\]
(as a map of sets for now) that sends $[I]$ to $[I\cap\kk[x,y]^\Gamma]$. Equivalently, if $V$ is the representation of the framed McKay quiver $Q$ of dimension vector $(1,n\delta)$ corresponding to $\kk[x,y]/I$, then $\tau(V)$ is the representation of $\mathcal{Q}$ from Figure~\ref{fig:quiverQ3loops} of dimension vector $(1,n)$ corresponding to $\kk[x,y]^\Gamma/(I\cap\kk[x,y]^\Gamma)$. Notice that $\tau$ forgets the information in $V$ encoded by $V_{k}$ for $1\leq k\leq r$, that is, $\tau$ satisfies 
\begin{equation}
    \label{eqn:tausets}
V=\kk\oplus \bigoplus_{0\leq k\leq r}V_k \mapsto \tau(V) =\kk \oplus V_0.
\end{equation}
The key point is that while $V$ is stable with respect to $\theta=(-n\delta,1,1,\dots,1)\in C_+$, we'll see that $\tau(V)$ encodes the data of a representation of $Q$ that is semistable with respect to 
\[
\theta_0:=(-n,1,0,\dots, 0)\in \overline{C_+}.
\]
The main result from
the author's joint work with S{\o}ren Gammelgaard, \'Ad\'am Gyenge and Bal\'{a}zs Szendr\H{o}i~
\cite{CGGS19} reconstructs $\tau$ by variation of GIT quotient and establishes the following:
 
\begin{theorem}[\textbf{$\Hilb^{[n]}(\mathbb{A}^2/\Gamma)$ as a quiver variety}]
\label{thm:main}
Let $\Gamma\subset \SL(2,\kk)$ be a finite subgroup and let $n>1$. Then for any $\theta\in C_+$ and for $\theta_0=(-n,1,0,\dots,0)\in \overline{C_+}$, there is a commutative diagram 
\[
\begin{tikzcd}
\mathfrak{M}_{\theta}(\mathbf{v},\mathbf{w})\ar[r] \ar[d,swap,"\sim"] &\ar[d, swap,"\sim"] \ar[r] \mathfrak{M}_{\theta_0}(\mathbf{v},\mathbf{w})\ar[d,swap] &\mathfrak{M}_0(\mathbf{v},\mathbf{w}) \ar[d,swap,"\sim"]
  \\
n\Gamma\text{-}\Hilb(\mathbb{A}^2) \ar[r,"\tau"]  &  \Hilb^{[n]}(\mathbb{A}^2/\Gamma)_{\red}\ar[r]  &  \Sym^n(\mathbb{A}^2/\Gamma)
 \end{tikzcd}
\]
 of morphisms over $\kk$ in which the vertical maps are isomorphisms and the horizontal maps along the top row are all induced by VGIT. In particular,  $\Hilb^{[n]}(\mathbb{A}^2/\Gamma)_{\red}$ is irreducible, normal, it has symplectic singularities, and $\tau$ provides a unique projective symplectic resolution.
\end{theorem}
 

\begin{remarks}
\begin{enumerate}
\item Irreducibility or otherwise of any scheme can be detected on the underlying reduced scheme, so we deduce that $\Hilb^{[n]}(\mathbb{A}^2/\Gamma)$ is irreducible.  
  \item As noted before, Bellamy--Schedler~\cite{BS21} establish that Nakajima quiver varieties $\mathfrak{M}_\eta(\mathbf{v},\mathbf{w})$ is irreducible, normal and it has symplectic singularities for any $\eta\in \Theta$, so the corresponding assertions for $\Hilb^{[n]}(\mathbb{A}^2/\Gamma)_{\red}$ follow from the middle isomorphism.
  \item Uniqueness of the symplectic resolution of $\mathfrak{M}_{\theta_0}(\mathbf{v},\mathbf{w})$ follows easily from \cite{BC20}. To see this, note that $\theta_0=(-n,1,0,\dots,0)$ lies in an extremal ray of the cone $F$ from \eqref{eqn:F}, while the decomposition of $F$ into GIT chambers given by Theorem~\ref{thm:GITdecomp} shows that $\theta_0$ lies in the closure of only one such chamber, namely $C_+$. The isomorphism of Theorem~\ref{thm:BC20main} now implies that the line bundle $L_F(\theta_0)$ lies in the closure of precisely one Mori chamber, namely, that defining $\mathfrak{M}_{\theta}(\mathbf{v},\mathbf{w})\cong n\Gamma\text{-}\Hilb(\mathbb{A}^2)$ for $\theta\in C_+$. Thus, the VGIT morphism $\mathfrak{M}_{\theta}(\mathbf{v},\mathbf{w})\to \mathfrak{M}_{\theta_0}(\mathbf{v},\mathbf{w})$ provides the unique projective symplectic resolution.
   \item The right-hand vertical isomorphism $\mathfrak{M}_0(\mathbf{v},\mathbf{w}) \cong \Sym^n(\mathbb{A}^2/\Gamma)$ is discussed in Remark~\ref{rems:Nakrems}(2).
\end{enumerate}
\end{remarks}

\begin{example}[\textbf{Type $A_2$ with $n=2$}]
 \label{exa:A2n2}
  The cyclic group $\Gamma\cong \ZZ/3$  is generated by the diagonal matrix $\text{diag}(\varepsilon,\varepsilon^2)$ for $\varepsilon = e^{2\pi i/3}$.  The $\Gamma$-invariant ideal $I=\langle x^6,y\rangle$ satisfies $\kk[x,y]/I = \kk\langle 1,x,x^2,x^3,x^4,x^5\rangle$, so it defines a point $[I]\in 2\Gamma\text{-}\Hilb(\mathbb{A}^2)$. The ideal $I\cap \kk[x,y]^\Gamma =\langle x^6, xy, y^3\rangle$
in the $\Gamma$-invariant subring satisfies $\kk[x,y]^\Gamma/(I\cap \kk[x,y]^\Gamma) = \kk\langle 1,x^3\rangle$, and $\tau([I]) = [I\cap \kk[x,y]^\Gamma]\in \Hilb^{[2]}(\mathbb{A}^2/\Gamma)$. 
  \end{example}

\section{On the proof of the main result}
\label{sec:sketch}
 There are two quite distinct steps in the proof of Theorem~\ref{thm:main}: the first is geometric; while the second is algebraic (and to a lesser extent, combinatorial) in nature.
 
 \subsection{Step 1: the geometry of linear series} 
 For $\theta\in C_+$, we first construct two morphisms from $\mathfrak{M}_\theta(\mathbf{v},\mathbf{w})$ and show that they agree up to the choice of scheme structure on $\Hilb^{[n]}(\mathbb{A}^2/\Gamma)$. The main player is the line bundle
 \[
 L:=L_F(\theta_0) = 
 \det(\mathcal{R}_{0})
 \]
 on $\mathfrak{M}_\theta(\mathbf{v},\mathbf{w})$ which, in light of Lemma~\ref{lem:ellCdet}, descends from the $G$-equivariant line bundle $\chi_{\theta_0}\otimes\mathcal{O}$ on $\mu^{-1}(0)^{\chi_\theta\textrm{-s}}$. Since $\theta_0$ lies in the boundary of the chamber $C_+$ containing $\theta$, the inclusion of $\mu^{-1}(0)^{\chi_\theta\textrm{-s}}$ into the $\theta_0$-semistable locus $\mu^{-1}(0)^{\chi_{\theta_0}\textrm{-ss}}$ is a $G$-equivariant map and, just as for the construction of $f_\theta$ from Remark~\ref{rems:Nakrems}(3), it induces a morphism
 \begin{equation}
     \label{eqn:varphiL}
 \varphi_{\vert L\vert }\colon \mathfrak{M}_\theta(\mathbf{v},\mathbf{w})\longrightarrow \mathfrak{M}_{\theta_0}(\mathbf{v},\mathbf{w}) =  \Proj \bigoplus_{k\geq 0} H^0\big(\mathfrak{M}_\theta(\mathbf{v},\mathbf{w}),L^{\otimes k}\big)\subseteq \vert L\vert
  \end{equation}
 by variation of GIT quotient that satisfies $L=\varphi_{\vert L\vert}^*\big(\mathcal{O}_{\mathfrak{M}}(1)\big)$.
 
 This same line bundle also features in the second morphism from $\mathfrak{M}_\theta(\mathbf{v},\mathbf{w})$ as follows.
 
 \begin{proposition}
 \label{prop:universal}
 The vector bundle $\mathcal{O}_{\mathfrak{M}}\oplus \mathcal{R}_{0}$ on $\mathfrak{M}_\theta(\mathbf{v},\mathbf{w})$ induces a universal morphism 
 \begin{equation}
     \label{eqn:tauvars} 
 \tau\colon\mathfrak{M}_\theta(\mathbf{v},\mathbf{w})\longrightarrow \Hilb^{[n]}(\mathbb{A}^2/\Gamma)
  \end{equation}
 that satisfies $L=\tau^*\big(\mathcal{O}_{\Hilb}(1)\big)$.
 \end{proposition} 
 
 Before proving this result, let's discuss the notion of \emph{cornering} a quiver algebra. First introduced in \cite{Craw11special} and studied in greater generality in joint work with Yukari Ito and Joe Karmazyn~\cite{CIK18}, the idea allows one to study the endomorphism algebra of a sub-bundle of the tautological bundle on a quiver moduli space. Recall that our quiver variety comes equipped with a tautological $\kk$-algebra homomorphism which, in light of Remark~\ref{rem:algebraA}, may be regarded as a $\kk$-algebra homomorphism
 \begin{equation}
 \label{eqn:tauthomoA}
 A\longrightarrow \End\bigg(\mathcal{O}_{\mathfrak{M}}\oplus \bigoplus_{0\leq k\leq r} \mathcal{R}_k\bigg)
 \end{equation}
 for $A\coloneqq \Pi/(b^*)$. Motivated by our discussion of the map of sets from \eqref{eqn:tausets}, our interest lies with the summand $\mathcal{O}_{\mathfrak{M}}\oplus \mathcal{R}_{0}$. To strip out the other summands from \eqref{eqn:tauthomoA}, let $e_k\in A$ denote the idempotent given by the class of the trivial path at vertex $k\in Q_0$; simplify notation by writing $e_k:=e_{\rho_k}$ for $0\leq k\leq r$. The process of \emph{cornering away}\footnote{
 If we view the vertices and paths defining elements of $A$ as the stations and routes in a railway system, then cornering away a vertex means that we're shutting the corresponding station but allowing trains to pass through it.} the vertices $\rho_1, \dots, \rho_r$ in $Q$ is the passage from $A$ to the subalgebra
$(e_\infty+e_{0}) A (e_\infty+e_{0})$ of $A$ spanned by the classes of paths whose tail and head lie in the set $\{\infty\}\cup \{\rho_0\}$. The next result is the fundamental reason why the main results of \cite{CGGS19} hold. 
 
 \begin{lemma}
 \label{lem:A0cornering}
 The algebra $A_0$ from \eqref{eqn:A0} is isomorphic to the subalgebra $(e_\infty+e_{0}) A (e_\infty+e_{0})$ of $A$. 
 \end{lemma}
 \begin{proof}
 Let $e_{0}\in \Pi_\Gamma\subset A$ denote the idempotent given by the class of the trivial path in the McKay quiver at vertex $\rho_0$. Lemma~\ref{lem:preprojEnd} gives a $\kk$-algebra isomorphism $e_{0} \Pi_\Gamma e_{0}\cong \End_{R_0}(R_0)\cong R_0$. The only paths in $Q$ that touch vertex $\infty$ are the trivial path $e_\infty$ and any path whose first arrow is $b$, so the algebra $(e_\infty+e_{0}) A (e_\infty+e_{0})$ is isomorphic to $R_0\oplus R_0b\oplus \kk e_\infty$ with multiplication given by 
 \[
 (\alpha_1,\beta_1b,\gamma_1e_\infty)\cdot (\alpha_2,\beta_2b,\gamma_2e_\infty) = \big(\alpha_1\alpha_2,(\alpha_1\beta_2+\beta_1\gamma_2)b,\gamma_1\gamma_2e_\infty\big).
 \]
 The result now follows from Lemma~\ref{lem:A0}.
 \end{proof}
  
  \begin{proof}[Proof of Proposition~\ref{prop:universal}]
   Start with the tautological $\kk$-algebra homomorphism \eqref{eqn:tauthomoA}, then multiply $A$ on left and right by $(e_\infty + e_{0})$, while simultaneously stripping away the summands $\mathcal{R}_{1}, \dots, \mathcal{R}_{r}$ from the endomorphism algebra. Applying Lemma~\ref{lem:A0cornering} gives a $\kk$-algebra homomorphism
   \[
   A_0\longrightarrow \End\big(\mathcal{O}_{\mathfrak{M}}\oplus  \mathcal{R}_0\big), 
   \]
   so the bundle $\mathcal{O}_{\mathfrak{M}}\oplus \mathcal{R}_{0}$ is a flat family of $A_0$-modules of dimension vector $(1,n)$ on $\mathfrak{M}_\theta(\mathbf{v},\mathbf{w})$. Each fibre of this family is $\infty$-generated, so Remark~\ref{rem:fine3} gives the universal morphism $\tau$ from \eqref{eqn:tauvars} such that $\tau^*(\mathcal{O}_{\Hilb})=\mathcal{O}_{\mathfrak{M}}$ and $\tau^*(\mathcal{T})=\mathcal{R}_{0}$. The polarising ample line bundle on $\Hilb^{[n]}(\mathbb{A}^2/\Gamma)$ is $\det(\mathcal{T})$, and since pullback commutes with tensor operations on bundles, we are able to conclude that $L=\det(\mathcal{R}_{0})\cong\det\big(\tau^*(\mathcal{T})\big)\cong  \tau^*(\det(\mathcal{T})\big) =\tau^*\big(\mathcal{O}_{\Hilb}(1)\big)$.
  \end{proof}

 The key line bundle $L$ can therefore be obtained as the pullback of the polarising ample line bundle via both $\varphi_{\vert L\vert}$ from \eqref{eqn:varphiL} and $\tau$ from \eqref{eqn:tauvars}, so the images of these morphisms coincide. It is not hard to show that $\varphi_{\vert L\vert}$ is surjective \cite[Lemma~2.2]{CGGS19}, so we obtain a commutative diagram
\[
\begin{tikzcd}
& \mathfrak{M}_\theta(\mathbf{v},\mathbf{w})\ar[dr,"\tau"] \ar[dl,swap,"\varphi_{\vert L\vert}"] & 
  \\
  \mathfrak{M}_{\theta_0}(\mathbf{v},\mathbf{w}) \ar[rr,"\iota"]  & &  \Hilb^{[n]}(\mathbb{A}^2/\Gamma)
 \end{tikzcd}
\]
where $\iota$ is a closed immersion.

 \subsection{Step 2: the algebra of recollement} 
 The isomorphism $\mathfrak{M}_{\theta_0}(\mathbf{v},\mathbf{w})\cong\Hilb^{[n]}(\mathbb{A}^2/\Gamma)_{\red}$ from Theorem~\ref{thm:main} will follow once we prove that $\iota$ is surjective on closed points. To this end, fix a closed point $[N]\in \Hilb^{[n]}(\mathbb{A}^2/\Gamma)$; note that Remark~\ref{rem:fine3} and Lemma~\ref{lem:A0} show that $N$ is an $\infty$-generated module over the algebra $A_0\cong (e_\infty+e_{0}) A (e_\infty+e_{0})$ obtained from $A$ by cornering away the vertices $\rho_1,\dots, \rho_r$.  This observation is important as it allows us to feed $N$ into the functor
 \[
 j_{!}\colon A_0\text{-mod}\longrightarrow A\text{-mod}\; : \; N\mapsto j_{!}(N)=A(e_\infty+e_{0})\otimes_{A_0} N
 \]
 that appears in a recollement of the abelian category of $A$-modules. The proof of the following result, which goes back to \cite[Lemma~3.6]{CIK18}, is straightforward, see \cite[Lemma~4.4-4.5]{CGGS19}:
 
 \begin{lemma}
 Let $N$ be an $\infty$-generated $A_0$-module of dimension vector $(1,n)$. The $A$-module $j_{!}(N)$ is $\theta_0$-semistable  of finite dimension such that $\dim_\infty j_{!}(N) = 1$ and $\dim_{0} j_{!}(N) = n$. 
 \end{lemma}
 
 At this stage, we don't yet know the 
 value of $\dim_{k} j_{!}(N)$ for $1\leq k\leq r$. If one could show that 
 \begin{equation}
     \label{eqn:ndeltainequal}
 \dim_{k} j_{!}(N)\leq n\dim(\rho_k) \quad\text{ for all }1\leq k\leq r,
  \end{equation}
 then the direct sum of $j_{!}(N)$ with $(n\dim(\rho_k)-\dim_{k} j_{!}(N))$ copies of each vertex simple $A$-module $\kk e_{k}$ for $1\leq k\leq r$ would define a $\theta_0$-semistable $A$-module $M$ of dimension vector $(1,n\delta)$ satisfying 
 \[
 (e_\infty+e_{0})A\otimes_{A_0}M \cong (e_\infty+e_{0})A\otimes_{A_0}j_{!}(N)\cong N.
 \]
 Then surjectivity of $\varphi_{\vert L\vert}$ and the construction of $\iota$ together give $\iota\big([M]\big)=[N]$, so $\iota$ (and hence also $\tau$) would be surjective on closed points, giving the desired isomorphism \[
 \mathfrak{M}_{\theta_0}(\mathbf{v},\mathbf{w})\cong\Hilb^{[n]}(\mathbb{A}^2/\Gamma)_{\red}.
 \]
 In fact, we don't know \eqref{eqn:ndeltainequal} in general, but it is nevertheless possible to establish inequalities analogous to those from \eqref{eqn:ndeltainequal} for the unique $\theta_0$-stable module appearing in the $\theta_0$-polystable representative in the S-equivalence class of $j_{!}(N)$; see \cite[Lemma~4.6, Section~A.1]{CGGS19}.
 
 We conclude by confirming that the proof of Theorem~\ref{thm:mainintro} is now complete.
 
\begin{proof}[Proof of Theorem~\ref{thm:mainintro}]
 The construction of the isomorphism $\mathfrak{M}_{\theta_0}(\mathbf{v},\mathbf{w})\cong\Hilb^{[n]}(\mathbb{A}^2/\Gamma)_{\red}$ required for Theorem~\ref{thm:mainintro}\one\ is described above. Parts \two\ and \three\ are given by the morphisms
 \[
\begin{tikzcd}
n\Gamma\text{-}\Hilb(\mathbb{A}^2)\cong \mathfrak{M}_{\theta}(\mathbf{v},\mathbf{w}) \ar[r,"\ref{thm:mainintro}\three"]  &  
\mathfrak{M}_{\theta_0}(\mathbf{v},\mathbf{w})\ar[r,"\ref{thm:mainintro}\two"]  &  \mathfrak{M}_{0}(\mathbf{v},\mathbf{w})\cong \Sym^n(\mathbb{A}^2/\Gamma)
 \end{tikzcd}
\]
 obtained by variation of GIT quotient.
\end{proof}

\subsection{Quot schemes for other GIT parameters}
Theorem~\ref{thm:mainintro} describes the Hilbert scheme of points on any ADE singularity $\mathbb{A}^2/\Gamma$ in terms of the Nakajima quiver variety $\mathfrak{M}_{\theta_0}(\mathbf{v},\mathbf{w})$ defined by the special stability parameter $\theta_0=(-n,1,0,\dots,0)\in \Theta$. However, the method introduced in the project \cite{CGGS19, CGGS21} works equally well for any stability parameter in the closure of $C_+$. 

 For any non-empty subset $I\subseteq \{0,1,\dots, r\}$, the stability parameter $\theta_I\in \Theta$ whose components equal 1 for each $i\in I$, and 0 otherwise, determines the key line bundle
 \[
 L_I:=L_F(\theta_I) = \bigotimes_{i\in I} \det(\mathcal{R}_{i})
 \]
 on $\mathfrak{M}_\theta(\mathbf{v},\mathbf{w})$ for $\theta\in C_+$. As we saw above in the special case $I=\{0\}$, there is a morphism
 \[
 \varphi_{\vert L_I\vert }\colon \mathfrak{M}_\theta(\mathbf{v},\mathbf{w})\longrightarrow \mathfrak{M}_{\theta_I}(\mathbf{v},\mathbf{w}) =  \Proj \bigoplus_{k\geq 0} H^0\Big(\mathfrak{M}_\theta(\mathbf{v},\mathbf{w}),L_I^{\otimes k}\Big)\subseteq \vert L_I\vert
  \]
 obtained by variation of GIT quotient that satisfies $L_I=\varphi_{\vert L_I\vert}^*\big(\mathcal{O}_{\mathfrak{M}}(1)\big)$. To obtain the analogue of the second morphism $\tau$, define the algebra
 \[
 A_I:=\Big(e_\infty+\textstyle{\sum_{i\in I}} e_{i}\Big) A \Big(e_\infty+\textstyle{\sum_{i\in I}} e_{i}\Big) 
 \]
 by cornering, then strip away the appropriate summands from the endomorphism algebra so that \eqref{eqn:tauthomoA} induces a $\kk$-algebra homomorphism
 \[
   A_I\longrightarrow \End\bigg(\mathcal{O}_{\mathfrak{M}}\oplus  \bigoplus_{i\in I} \mathcal{R}_{i}\bigg).
   \]
  Thus, the bundle $\mathcal{O}_{\mathfrak{M}}\oplus \bigoplus_{i\in I}\mathcal{R}_{i}$ on $\mathfrak{M}_\theta(\mathbf{v},\mathbf{w})$ is a flat family of $\infty$-generated $A_I$-modules of dimension vector $(1,n_I)$ for $n_I:=\sum_{i\in I}n\dim(\rho_i)\rho_i$. A special case of \cite[Proposition~4.2]{CGGS21} establishes that the fine moduli space of $\infty$-generated $A_I$-modules of dimension vector $n_I$ is isomorphic to a particular \emph{orbifold Quot scheme}; as a set, this is   
   \[
   \Quot_I^{n_I}\big([\mathbb{A}^2/\Gamma]\big):= \left\{\End\bigg(\bigoplus_{i\in I}R_{i}\bigg) \text{-epimorphisms }\bigoplus_{i\in I} R_{i}\twoheadrightarrow Z \mathrel{\Big|}  \dim e_{i}Z = n\dim(\rho_i)\text{ for }i\in I\right\}.
   \]
   For $I=\{0\}$, we have $R_0\cong \kk[x,y]^\Gamma$ and $n_I=n$, so we recover the set from \eqref{eqn:HilbnADE}. A proof similar to that described above for Proposition~\ref{prop:universal} establishes the following:
   
 \begin{proposition}
 \label{prop:universalquot}
 The vector bundle $\mathcal{O}_{\mathfrak{M}}\oplus \bigoplus_{i\in I}\mathcal{R}_{i}$ on $\mathfrak{M}_\theta(\mathbf{v},\mathbf{w})$ induces a universal morphism 
 \begin{equation}
     \label{eqn:tauvars2} 
 \tau_I\colon\mathfrak{M}_\theta(\mathbf{v},\mathbf{w})\longrightarrow \Quot_I^{n_I}\big([\mathbb{A}^2/\Gamma]\big)
  \end{equation}
 that satisfies $L_I=\tau_I^*\big(\mathcal{O}_{\Quot}(1)\big)$.
 \end{proposition} 
   
The approach described above for the case $I=\{0\}$ generalises to give a commutative diagram
\[
\begin{tikzcd}
& \mathfrak{M}_\theta(\mathbf{v},\mathbf{w})\ar[dr,"\tau_I"] \ar[dl,swap,"\varphi_{\vert L_I\vert}"] & 
  \\
  \mathfrak{M}_{\theta_I}(\mathbf{v},\mathbf{w}) \ar[rr,"\iota_I"]  & &  \Quot^{n_I}_I\big([\mathbb{A}^2/\Gamma]\big)
 \end{tikzcd}
\]
where $\iota_I$ is a closed immersion. The harder part is to show that $\iota_I$ is surjective on closed points; at least, it's hard when $0\not\in I$. Again, $\varphi_{\vert L_I\vert}$ is surjective for our special choice of dimension vector $n_I$, and the same algebraic approach via a recollement on $A\text{-mod}$ gives the required surjectivity statement. 
For details, one can appeal to \cite[Proposition~4.2]{CGGS21} and invoke the case-by-case analysis of \cite[Appendix~A]{CGGS19}; alternatively,  \cite{CGGS21} bypasses the case-by-case analysis at the expense of replacing $\mathbf{v}$ by an alternative dimension vector as in \cite[Theorem~6.9]{CGGS21}. 

The results of this section can be summarised in the following statement generalising Theorem~\ref{thm:mainintro}.

 \begin{theorem}[\textbf{Orbifold Quot schemes as quiver varieties}]
 \label{thm:mainrevisited}
  Let $\Gamma\subset \SL(2,\kk)$ be a finite subgroup and let $n\geq 1$. For any non-empty subset $I\subseteq \{0,1,\dots,r\}$ and for $n_I:=\sum_{i\in I}n\dim(\rho_i)\rho_i$, the orbifold Quot scheme $\Quot_I^{n_I}\big([\mathbb{A}^2/\Gamma]\big)_{\red}:$
  \begin{enumerate}
      \item[\one] is a Nakajima quiver variety for some 
      GIT stability parameter in $\overline{C_+}$;
      \item[\two] is a partial crepant resolution of the affine variety $\Sym^n(\mathbb{A}^2/\Gamma)$; and
      \item[\three] admits a projective crepant resolution obtained by variation of GIT quotient.
  \end{enumerate}
\end{theorem}

 \subsection{The elephant in the room}
 It would be inappropriate to wrap up without at least addressing the question of reducedness.
 
 \begin{question}
 Is the orbifold Quot scheme $\Quot_I^{n_I}\big([\mathbb{A}^2/\Gamma]\big)$ reduced? How about $\Hilb^{[n]}(\mathbb{A}^2/\Gamma)$?
 \end{question}

 Unfortunately, our methods don't shed light on this question because we build the map that provides the set-theoretic inverse to $\iota_I$ on closed points by applying the functor $j_{!}$ one module at a time. A more global construction, in which $j_{!}$ is applied uniformly to a flat family of $A_I$-modules, would be required to construct an inverse morphism in the algebraic category.

\end{document}